\documentclass[journal,onecolumn,draftclsnofoot]{IEEEtran}
\usepackage{multirow}
\usepackage{graphicx}
\usepackage{amsmath}
\usepackage{amsthm}
\usepackage{mathtools}
\usepackage[psamsfonts]{amssymb}
\usepackage{amsxtra}
\usepackage{hyperref}
\usepackage{paralist}
\usepackage{threeparttable}
\usepackage{cleveref}
\usepackage{cite}
\usepackage{color}
\usepackage[justification=centering]{caption}

\newtheorem{Theorem}{Theorem}[section]
\newtheorem{Lemma}[Theorem]{Lemma}

\newtheorem{Remark}[Theorem]{Remark}
\theoremstyle{definition}
\newtheorem{assump}{Assumption}

\newenvironment{myassump}[2][]
{\renewcommand\theassump{#2}\begin{assump}[#1]}
	{\end{assump}}

\allowdisplaybreaks

\def \btheta{\boldsymbol{\theta}}

\title{$\mathcal{CIRFE}$: A Distributed Random Fields Estimator}
\author{Anit~Kumar~Sahu,~\IEEEmembership{Student Member,~IEEE}, Dusan~Jakovetic,~\IEEEmembership{Member,~IEEE} and Soummya~Kar,~\IEEEmembership{Member,~IEEE}
	\thanks{
		A. K. Sahu and S. Kar are with the Department of Electrical and Computer Engineering, Carnegie Mellon University, Pittsburgh, PA 15213, USA (email:anits@andrew.cmu.edu, soummyak@andrew.cmu.edu). D. Jakovetic is with the Department of Mathematics and Informatics, University of Novi Sad, Serbia (e-mail: djakovet@uns.ac.rs).
		
		The work of A. K. Sahu and S. Kar was supported in part by NSF under grants CCF-1513936. The work of D. Jakovetic was supported by Ministry of Education, Science and
		Technological Development, Republic of Serbia, grant no. 174030.
	}}
	
\begin{document}

		\maketitle
		\thispagestyle{empty}
		\pagestyle{empty}

		%%%%%%%%%%%%%%%%%%%%%%%%%%%%%%%%%%%%%%%%%%%%%%%%%%%%%%%%%%%%%%%%%%%%%%%%%%%%%%%%
		\begin{abstract}
			\noindent This paper presents a communication efficient distributed algorithm, $\mathcal{CIRFE}$ of the \emph{consensus}+\emph{innovations} type, to estimate a high-dimensional parameter in a multi-agent network, in which each agent is interested in reconstructing only a few components of the parameter. This problem arises for example when monitoring the high-dimensional distributed state of a large-scale infrastructure with a network of limited capability sensors and where each sensor is tasked with estimating some local components of the state. At each observation sampling epoch, each agent updates its local estimate of the parameter components in its interest set by simultaneously processing the latest locally sensed information~(\emph{innovations}) and the parameter estimates from agents~(\emph{consensus}) in its communication neighborhood given by a time-varying possibly sparse graph. Under minimal conditions on the inter-agent communication network and the sensing models, almost sure convergence of the estimate sequence at each agent to the components of the true parameter in its interest set is established. Furthermore, the paper establishes the performance of $\mathcal{CIRFE}$ in terms of asymptotic covariance of the estimate sequences and specifically characterizes the dependencies of the component wise asymptotic covariance in terms of the number of agents tasked with estimating it. Finally, simulation experiments demonstrate the efficacy of $\mathcal{CIRFE}$.
			
		\end{abstract}
		
		
		%%%%%%%%%%%%%%%%%%%%%%%%%%%%%%%%%%%%%%%%%%%%%%%%%%%%%%%%%%%%%%%%%%%%%%%%%%%%%%%%
		\section{INTRODUCTION}
		%\noindent Large scale cyber-physical systems~(CPS) have gained prominence recently. A desirable trait of a CPS is to be energy efficient so as to have a longer operational life. Owing to the the deployment of such systems in a large-scale fashion spread over a large area, the sensing and the data collection are inherently distributed in such systems. Due to the large-scale deployment of such systems compounded with the inherent distributed nature of the data collection, a centralized coordinator based architecture is undesirable. A typical entity of a CPS with on board sensing and computational capabilities is plagued by power hungry communication needs. Also, relevant to a large-scale CPS is the associated high-dimensional state of the CPS, which the network as a whole wants to reconstruct. Moreover, the high-dimensional state reconstruction needs to be done in a recursive fashion, which in essence points to the fact that the state estimate needs to be updated at each entity as and when a new datum is collected. The high-dimensional state of the CPS in conjunction with the power hungry communication bottleneck makes the problem at hand very challenging.
		\noindent In this paper, we are interested in distributed inference of the parameterized state of the large-scale cyber physical systems~(CPS) like sensor networks monitoring a spatially distributed field, or CPSs where physical entities with sensing capabilities are deployed over large areas. Relevant applications include, minimum cost flow problems~(see, for example \cite{ahuja1988network}), distributed model predictive control~(see, for example \cite{mota2015distributed,halvgaard2016distributed}), distributed localization~(see, for example \cite{Khan-DILAND-TSP-2010}). An important example of the systems of interest is the smart grid--a large network of generators and loads instrumented with, for example, phasor measurement units (PMUs)\cite{kekatos2013distributed,de2010synchronized}. Our goal is to reconstruct the physical field or state of the CPS that is represented by a vector parameter. The structure of the physical layer is reflected through the coupling among the observation sequences across different nodes. Suppose, for the purpose of illustration, corresponding to each field location, there is a low-power inexpensive sensor monitoring the location. The noisy sensor measurement at a location in the field is possibly a function of its own component and \emph{neighboring} field components. As an example, in the smart grid context, a sensor at a node (location) may obtain a measurement of the power flowing into that node, which in turn is a function of the field components (e.g., voltages, angles) at that node and neighboring nodes. This coupling among parameter components in the measurements will be referred to as the \emph{physical coupling} in the sequel. However, due to possible lack of identifiability, in order to come up with a provably consistent estimates of the parameter components of interest, each agent exchanges information with its neighborhood which conforms to a pre-assigned inter-agent communication graph. The inter-agent communication graph forms the cyber layer of the system and is different from that of the physical layer, i.e., the coupling structure among the parameter components induced by the distributed measurement model. Due to the high-dimensionality of the field, reconstructing the entire field at each agent may be too taxing, and hence, agents may only be interested in estimating certain components of the parameter field locally; furthermore, the components of interest at a given agent, referred to as the \emph{interest set} of the agent, varies from agent to agent. More concretely, the observation model we adopt in this paper is of the form,
		{\small\begin{align*}
			\mathbf{y}_{n}(t) = \mathbf{H}_{n}\btheta^{\ast}+\mathbf{\gamma}_{n}(t),
			\end{align*}}
		where $\btheta^{\ast}\in\mathbb{R}^{N}$, and the dimension of $\btheta^{\ast}$ corresponds to the number of physical locations being monitored in the field, $\mathbf{H}_{n}$ is a~(fat) matrix that abstracts the coupling in measurements for agent $n$ with the state values at other nodes and $\mathbf{\gamma}_{n}(t)$ represents the observation noise~(to be specified later). Existing distributed estimation schemes, such as in \cite{kar2011convergence,kar2013consensus+,Mesbahi-parameter,Giannakis-est,Sayed-LMS,Stankovic-parameter,Giannakis-LMS,Nedic-parameter,sahu2016distributed}, aim to reconstruct the entire parameter at each node of the networked setup, thus reflecting a homogeneous objective across all nodes. However, in this paper, we consider a distributed estimation scheme which allows agents to pursue a heterogeneous objectives, in which agents' only estimate a few components of the parameter vector $\btheta^{\ast}$ corresponding to their \emph{interest sets}. Accounting for heterogeneity is highly relevant in practice. The heterogeneous objectives across agents lets us extend the notion of \emph{consensus} to \emph{subspace consensus}, where the agents reach consensus with respect to entries of the parameter that lie in the intersection of their interest sets. The second level of heterogeneity in our proposed algorithm is exhibited in terms of heterogeneous agent sensing models and noise statistics. In practice, different types of devices~(agents) in the network may have very different ``sensing quality''. For instance, with state estimation in smart grids, phasor measurement units (PMUs) can have much smaller variance than standard sensing devices.
		Owing to the high-dimensionality of the state vector and limited storage and processing capabilities in the individual entities of a large-scale CPS, exchanging high-dimensional estimates may be undesirable. Hence, in this paper, we consider a distributed estimator where each agent only infers a fraction of the field, but through cooperation and under the appropriate conditions generates provably consistent estimates of this fraction of the field. In existing distributed estimation schemes such as in \cite{kar2011convergence,kar2013consensus+,Mesbahi-parameter,Giannakis-est,Sayed-LMS,Stankovic-parameter,Giannakis-LMS,Nedic-parameter,sahu2016distributed}, the global model information in terms of the sensing models of all the agents are assumed to be inaccessible for any agent. However, the aforementioned setups subsume the knowledge of the dimension of the state vector to be estimated and hence adapt the storage requirements at each agent to cater to the exact dimension of the state vector. In contrast, in this paper, we present a distributed estimation algorithm of the $\emph{consensus}+\emph{innovations}$ form~(\hspace{-0.5pt}\cite{kar2011convergence,kar2013consensus+}), namely $\mathcal{CIRFE}$, \emph{consensus+innovations} Random Fields Estimator, where each agent reconstructs only a subset of the field by simultaneously processing information obtained from its neighbors~(\emph{consensus}) and the latest sensed information~(\emph{innovation}). It is of particular interest, that the information about the state vector is constrained to neighborhoods; in that, an agent has only information about the components of the state vector which potentially affect its own measurements, alleviating storage and model knowledge requirements.\\
		Our main contributions are as follows:\\
		\noindent \textbf{Main Contribution 1:}~We propose a scheme, namely $\mathcal{CIRFE}$,  where each entity reconstructs only a subset of the components of the state modeled by a vector parameter, and thereby also reducing the dimension of messages being communicated among the agents. Under mild conditions of the connectivity of the network, we establish consistency of the estimate sequence at each agent with respect to the components of the parameters in its interest set. The proposed scheme allows heterogeneity in terms of agents' objectives, while still allowing for inter-agent collaboration. \par
		\noindent \textbf{Main Contribution 2:}~Technically, the consensus+innovations type approach that we employ for state reconstruction in the current setting constitutes a mixed time-scale stochastic approximation procedure \cite{nevelson1973stochastic}. We explicitly evaluate the asymptotic covariance of the component wise estimate sequences at each agent. The obtained asymptotic covariance is heterogeneous in terms of scaling of the variances of the components of the parameter based on the number of agents interested in reconstructing a particular component. To the best of our knowledge, this is the first asymptotic covariance evaluation explicitly in terms of the number of agents interested in reconstructing entries of the state vector for distributed estimation of high-dimensional fields, i.e., when each node is interested only in a subset of the vector parameter. \\
		\noindent\textbf{Related Work:} Relevant distributed estimation literature can be classified primarily into two types. The first type includes schemes which involve single snapshot data collection followed by inter-agent fusion through consensus type protocols~(see, for example,\cite{Mesbahi-parameter,Giannakis-est,Bertsekas-survey,olfatisaberfaxmurray07,jadbabailinmorse03}). The second type includes estimation schemes where the sensing and the processing of the information occur at the same rate and sequentially in time~(see, for example \cite{Sayed-LMS,Stankovic-parameter,Giannakis-LMS,Nedic-parameter,nedic2015nonasymptotic,jadbabaie2012non,shahrampour2013exponentially,mateos2009distributed,mateos2012distributed,weng2014efficient}). Representative approaches of this latter class are $\emph{consensus}+\emph{innovations}$ type~(\hspace{-0.5pt}\cite{kar2011convergence,sahu2016distributed})~and the \emph{diffusion} type~(\hspace{-0.5pt}\cite{cattivelli2010diffusion,cattivelli2008diffusion}) algorithms. Distributed inference algorithms for random fields have been proposed in literature, see, for example \cite{Khan-Moura,das2015distributed}. Reference \cite{Khan-Moura} considers the estimation of a time-varying random field pertaining to a linear observation model, where each agent reconstructs only a few components of the field. However, in contrast with \cite{Khan-Moura}~where the incorporation of new sensed information is followed by multiple rounds of consensus, the proposed algorithm in this paper simultaneously fuses the neighborhood information and the current observation albeit for a static field. In contrast with \cite{das2015distributed}, where each agent tries to reconstruct the entire time-varying random field, the proposed algorithm reconstructs only a subset of the components of the entire field at each agent and the information exchange entails a low dimensional vector instead of the entire parameter. Distributed estimation schemes involving objectives where agents reconstruct only a few entries of the parameter have also been studied in \cite{kekatos2013distributed,bogdanovic2014distributed,nassif2017diffusion}. In particular, as compared to  \cite{kekatos2013distributed,bogdanovic2014distributed,nassif2017diffusion} which consider static connected communication graphs, in this paper we consider time-varying stochastic communication graphs that are connected on average. In \cite{nassif2017diffusion} so as to facilitate adaptation, the algorithm employs constant step sizes and the residual mean square error is characterized in terms of the step size only. However in comparison, the asymptotic variance of the estimator proposed in this paper reveals the scaling with respect to the number of agents interested in reconstructing a particular entry of the parameter. A field estimation scheme in a fully distributed setup of the type studied in this paper with arbitrary connected inter-agent communication topology where agents reconstruct only a subset of the physical field was also proposed in \cite{kar2010large}~(Chapter 3). The current work is inspired by \cite{kar2010large} and generalizes the development in \cite{kar2010large} in several fronts to achieve better estimation performance. In \cite{kar2010large}, a \emph{single time-scale} consensus+innovations algorithm pertaining to a linear observation model was proposed and the consistency and asymptotic normality\footnote{An estimate sequence is asymptotically normal if its $\sqrt{t}$ scaled error process, i.e., the difference between the sequence and the true parameter converges in distribution to a normal random variate.} of the estimator was established. By, a \emph{single time-scale} consensus+innovations algorithm, we mean algorithms where the consensus and innovation potentials are controlled by the same time-decaying sequence. The performance of the single time-scale version of the consensus+innovations distributed estimation algorithm in terms of asymptotic variance depends on the network topology and is thus affected when the connectivity of the network is relatively poor. In contrast with \cite{kar2010large}, we propose a $\emph{consensus}+\emph{innovations}$ algorithm, where the \emph{consensus} and \emph{innovations} terms are weighed through different carefully crafted time-varying sequences. In this paper, we not only establish the consistency and asymptotic normality of the parameter estimate sequence but also, due to the employed mixed time-scale stochastic approximation obtain the asymptotic covariance of the estimate sequences to be independent of the particular communication network instance.\\
		\noindent\textbf{Paper Organization :}
		The rest of the paper is organized as follows. Spectral graph theory and notation are discussed next. The sensing model and the preliminaries are discussed in Section \ref{sec:sens_prel}. Section \ref{sec:cirfe} presents the proposed distributed estimation algorithm, while Section \ref{sec:main_res} and Section \ref{sec:proof_mainres} concerns with the main results of the paper and the proof of the main results respectively. The simulation experiments for the proposed algorithm are presented in Section \ref{sec:sim}. Finally, Section \ref{sec:conc} concludes the paper.\\
		\noindent\textbf{Notation.}
		\noindent We denote by~$\mathbb{R}$ the set of reals,  and by~$\mathbb{R}^{k}$ the $k$-dimensional Euclidean space. Vectors and matrices are in bold faces. We also denote by $\mathbf{A}_{ij}$ or $[\mathbf{A}]_{ij}$, the $(i,j)$-th entry of a matrix $\mathbf{A}$; $\mathbf{a}_{i}$ or $[\mathbf{a}]_{i}$ the $i$-th entry of a vector $\mathbf{a}$. The symbols $\mathbf{I}$ and $\mathbf{0}$ are the $k\times k$ identity matrix and the $k\times k$ zero matrix, respectively, the dimensions being clear from the context. The vector $\mathbf{e_{i}}$ is the $i$-th column of $\mathbf{I}$, also referred to as a canonical vector. The symbol $\top$ stands for matrix transpose. The operator $\otimes$ denotes the Kronecker product. The operator $|| . ||$ applied to a vector is the standard Euclidean $\mathcal{L}_{2}$ norm, while when applied to matrices stands for the induced $\mathcal{L}_{2}$ norm, which is equal to the spectral radius for symmetric matrices. The cardinality of a set $\mathcal{S}$ is $\left|\mathcal{S}\right|$. Finally, $diag(\mathbf{v})$ denotes the diagonal matrix with its diagonal elements as $\mathbf{v}$.
		All inequalities involving random variables are to be interpreted almost surely~(a.s.). \\
		\noindent\textbf{Spectral Graph Theory.}
		\noindent The inter-agent communication network is a simple\footnote{A graph is said to be simple if it is devoid of self loops and multiple edges.} undirected graph $G=(V, E)$, where $V$ denotes the set of agents or vertices with cardinality $|V|=N$, and $E$ the set of edges with $|E|=M$. If there exists an edge between agents $i$ and $j$, then $(i,j)\in E$. A path between agents $i$ and $j$ of length $m$ is a sequence ($i=p_{0},p_{1},\cdots,p_{m}=j)$ of vertices, such that $(p_{t}, p_{t+1})\in E$, $0\le t \le m-1$. A graph is connected if there exists a path between all possible agent pairs.
		The neighborhood of an agent $n$ is given by $\Omega_{n}=\{j \in V|(n,j) \in E\}$.
		\noindent The degree of agent $n$ is given by $d_{n}=|\Omega_{n}|$. The structure of the graph is represented by the symmetric $N\times N$ adjacency matrix $\mathbf{A}=[A_{ij}]$, where $A_{ij}=1$ if $(i,j) \in E$, and $0$ otherwise. The degree matrix is given by the diagonal matrix $\mathbf{D}=diag(d_{1}\cdots d_{N})$. The graph Laplacian matrix is defined as $\mathbf{L}=\mathbf{D}-\mathbf{A}$.
		\noindent The Laplacian is a positive semidefinite matrix, hence its eigenvalues can be ordered and represented as $0=\lambda_{1}(\mathbf{L})\le\lambda_{2}(\mathbf{L})\le \cdots \lambda_{N}(\mathbf{L})$.
		\noindent Furthermore, a graph is connected if and only if $\lambda_{2}(\mathbf{L})>0$ (see~\cite{chung1997spectral}~for instance).\\
		\section{SENSING MODEL AND PRELIMINARIES}
		\label{sec:sens_prel}
		\noindent Consider $N$ physical agents monitoring a field over a large physical area. Each agent $n$ is associated with a scalar state $\theta_{n}^{\ast}$, which represents the field intensity parameter at its location. The agents are equipped with sensing capabilities. We assume each agent observes a time-series of measurements, given by noisy linear functions of its state and the states of \emph{neighboring} agents. Due to this coupling in the observations, an agent should cooperate with neighbors to reconstruct its own state. For simplicity, we assume that the individual agent states are scalars. Our results can be generalized to vector valued states, though at the cost of extra notation. The observation at each agent is of the form:
		{\small\begin{align}
			\label{RLU-prob1}
			\mathbf{y}_{n}(t)=\mathbf{H}_{n}\mathbf{\btheta}^{\ast}+\mathbf{\gamma}_{n}(t),
			\end{align}}
		\noindent where $\mathbf{H}_{n}\in\mathbb{R}^{M_{n}\times N}$ is a sparsifying~(to be clarified soon) sensing matrix, $\{\mathbf{y}_{n}(t)\}$ is a $\mathbb{R}^{M_{n}}$-valued observation sequence for the $n$-th agent and for each $n$ where possibly $M_{n} \ll  N$, $\left\{\mathbf{\gamma}_{n}(t)\right\}$ is a zero-mean temporally independent and identically distributed~(i.i.d.)~ noise sequence with nonsingular covariance matrix $\mathbf{R}_{n}$. It is to be noted that the assumption that the dimension of the parameter $\boldsymbol{\theta}^{\ast}$ is equal to the number of agents, $N$, is simply made for clarity of presentation. In particular, all our proofs and assertions will continue to hold with appropriate modifications if the dimension of the global parameter is different from $N$.
		\begin{myassump}{A1}
			\label{as:1}
			\emph{There exists $\epsilon_{1}>0$, such that, for all $n$, $\mathbb{E}_{\btheta}\left[\left\|\gamma_{n}(t)\right\|^{2+\epsilon_{1}}\right]<\infty$.}
		\end{myassump}
		\noindent The above assumption encompasses a broad class of noise distributions in the setup. The heterogeneity of the setup is exhibited in terms of the sensing matrix and the noise covariances at the agents. We now formalize an assumption on global model observability.\\
		\begin{myassump}{A2}
			\label{as:2}
			\emph{The matrix $\mathbf{G}=\sum_{n=1}^{N}\mathbf{H}_{n}^{\top}\mathbf{R}_{n}^{-1}\mathbf{H}_{n}$ is full rank}.
		\end{myassump}
		\noindent Assumption \ref{as:2} is crucial for our distributed setup. It is to be noted that such an assumption is needed for even a setup with a centralized node which has access to all the data samples at each of the agent nodes at each time. Assumption \ref{as:2} ensures that if a hypothetical fusion center could stack all the data samples together at any time $t$, it would have sufficient information so as to be able to unambiguously estimate the parameter of interest. Hence, the requirement for this assumption naturally extends to our distributed setup.
		\noindent As far as reconstructing the parameter $\btheta$ is concerned, there is an inherent scalability issue as the dimension of the parameter scales with the size of the network. Owing to the ad-hoc nature of setups as described above and observations being made at different agents in a sequential manner, one has to resort to recursive message-passing schemes while conforming to a communication protocol specified by a inter-agent communication graph. Given the possibly high-dimensional state of the field, it is not desirable and communication-wise feasible to exchange the high-dimensional data in the form of parameter estimates and for each agent to estimate the entire vector. Before, going over specifics of our algorithm, we next review recursive estimation both in the centralized and distributed setups.
		\subsection{Preliminaries}
		\label{subsec:prel}
		\noindent  In this section, we go over the preliminaries of classical distributed estimation.\\
		%\textbf{Oracle Estimation:}\\
		%In the setup described above in \eqref{RLU-prob1}, if a hypothetical oracle node having access to the data samples at all the nodes at all times were to conduct the parameter estimation in an iterative manner, it would do so in the following way:
		%\begin{align*}
		%&\mathbf{x}_{c}(t+1)=\mathbf{x}_{c}(t)\nonumber\\&+\underbrace{\frac{a}{t+1}\sum_{n=1}^{N}\mathbf{H}_{n}^{\top}\mathbf{\Sigma}_{n}^{-1}\left(\mathbf{y}_{n}(t)-\mathbf{H}_{n}\mathbf{x}_{c}(t)\right)}_{\text{Global Innovation}},
		%\end{align*}
		%where $a$ is a positive constant. It is well known from standard stochastic approximation results that the sequence $\{\mathbf{x}_{c}(t)\}$ generated from the update above converges almost surely to the true parameter $\btheta$. Moreover, the sequence $\{\mathbf{x}_{c}(t)\}$ is asymptotically normal, i.e,
		%\begin{align*}
		%\sqrt{t+1}\left(\mathbf{x}_{c}(t)-\btheta\right)\overset{\mathcal{D}}{\Longrightarrow}\mathcal{N}\left(0,\left(N\mathbf{\Gamma}\right)^{-1}\right),
		%\end{align*}
		%where $\mathbf{\Gamma}=\frac{1}{N}\sum_{n=1}^{N}\mathbf{H}_{n}^{\top}\mathbf{\Sigma}_{n}^{-1}\mathbf{H}_{n}$. The above established asymptotic normality also points to the conclusion that the mean square error~(MSE) decays as $\Theta(1/t)$.\\
		%However, such an oracle based scheme may not be implementable in our distributed multi-agent setting with time-varying sparse inter-agent interaction primarily due to the fact that the desired global innovation computation requires instantaneous access to the entire set of network sensed data at all times at the oracle.\\
		\noindent \textbf{Distributed Estimation:}\\
		\noindent In the setup described above in \eqref{RLU-prob1}, if a hypothetical fusion center having access to the data samples at all nodes at all times were to conduct the parameter estimation in a recursive manner, a (centralized) recursive least-squares type approach could be employed as follows:
		{\small\begin{align*}
			&\mathbf{x}_{c}(t+1)=\mathbf{x}_{c}(t)\nonumber\\&+\underbrace{\frac{a}{t+1}\sum_{n=1}^{N}\mathbf{H}_{n}^{\top}\mathbf{\Sigma}_{n}^{-1}\left(\mathbf{y}_{n}(t)-\mathbf{H}_{n}\mathbf{x}_{c}(t)\right)}_{\text{Global Innovation}},
			\end{align*}}
		where $a$ is a positive constant such that $a> N/\left(\lambda_{min}\left(\sum_{n=1}^{N}\mathbf{H}_{n}^{\top}\mathbf{R}_{n}^{-1}\mathbf{H}_{n}\right)\right)$.
		\noindent However, such a fusion center based scheme may not be implementable in our distributed multi-agent setting with time-varying sparse inter-agent interaction primarily due to the fact that the desired global innovation computation requires instantaneous access to the entire set of network sensed data at all times at the fusion center. Moreover, the fusion center intends to reconstruct the entire high-dimensional state and thus, maintains a $N$-dimensional estimate at all times.
		\noindent If in the case of a distributed setup, an agent $n$ in the network were to replicate the centralized update by replacing the global innovation in accordance with its local innovation, the update for the parameter estimate becomes
		{\small\begin{align*}
			&\widehat{\mathbf{x}}_{n}(t+1)=\widehat{\mathbf{x}}_{n}(t)\nonumber\\&+\underbrace{\frac{a}{t+1}\mathbf{H}_{n}^{\top}\mathbf{\Sigma}_{n}^{-1}\left(\mathbf{y}_{n}(t)-\mathbf{H}_{n}\widehat{\mathbf{x}}_{n}(t)\right)}_{\text{Local Innovation}},
			\end{align*}}
		\noindent where $\left\{\widehat{\mathbf{x}}_{n}(t)\right\}$ represents the estimate sequence at agent $n$. The above update involves purely decentralized and independent local processing with no collaboration among the agents whatsoever. However, note that in the case when the data samples obtained at each agent lacks information about all the features, the parameter estimates would be erroneous and sub-optimal. As in the case of the fusion center based approach outlined above, each agent maintains a $N$-dimensional estimate at all times and hence the messages exchanged in the neighborhood are $N$-dimensional and could be very large depending on the size of the network.
		\noindent Hence, as a surrogate to the global innovation in the centralized recursions, the local estimators compute a local innovation based on the locally sensed data as an agent has access to information only in its neighborhood. The information loss at a node is compensated by incorporating an agreement or consensus potential into their updates which is then incorporated~(see, for example \cite{kar2011convergence,kar2013distributed,sahu2016distributedtsipn}) as follows:
		{\small\begin{align}
			\label{eq:benchmark_ci}
			&\mathbf{x}_{n}(t+1) = \mathbf{x}_{n}(t)-\underbrace{\frac{b}{(t+1)^{\delta_{1}}}\sum_{l\in\Omega_{n}(t)}\left(\mathbf{x}_{n}(t)-\mathbf{x}_{l}(t)\right)}_{\text{Neighborhood Consensus}}\nonumber\\&+\underbrace{\frac{a}{t+1}\mathbf{\Gamma}^{-1}\mathbf{H}_{n}^{\top}\mathbf{\Sigma}_{n}^{-1}\left(\mathbf{y}_{n}(t)-\mathbf{H}_{n}\mathbf{x}_{n}(t)\right)}_{\text{Local Innovation}},
			\end{align}}
		\noindent where $0 <\delta_{1}<1$, $\Omega_{n}(t)$ represents the neighborhood of agent $n$ at time $t$ and $a,b$ are appropriately chosen positive constants. In the above scheme, the information exchange among agent nodes is limited to the parameter estimates. It has been shown in previous work that under appropriate conditions~(see, for example \cite{kar2011convergence}), the estimate sequence  $\{\mathbf{x}_{n}(t)\}$ converges to $\btheta^{\ast}$ and is asymptotically normal, i.e.,
		{\small\begin{align*}
			\sqrt{t+1}\left(\mathbf{x}_{n}(t)-\btheta\right)\overset{\mathcal{D}}{\Longrightarrow}\mathcal{N}\left(0,\left(N\mathbf{\Gamma}\right)^{-1}\right),
			\end{align*}}
		where $\mathbf{\Gamma}=\frac{1}{N}\sum_{n=1}^{N}\mathbf{H}_{n}^{\top}\mathbf{R}_{n}^{-1}\mathbf{H}_{n}$ and $\overset{\mathcal{D}}{\Longrightarrow}$ denotes convergence in distribution. \noindent The above established asymptotic normality also points to the conclusion that the MSE decays as $\Theta(1/t)$. For future reference, we will refer to the distributed estimation approach in \eqref{eq:benchmark_ci} as the classical consensus+innovations approach. The aforementioned scheme, though optimal in terms of the asymptotic covariance entails the availability of global model information at each agent and exchange of the entire parameter estimate which in turn is $N$-dimensional among agents. Furthermore, due to the inherent spatial coupling in the observation sequence at each node with other nodes in its neighborhood, the availability of a particular entry of the state vector is localized to a small area. Hence, a large-scale deployment of such a system, would incorporate a significant delay for an agent to assimilate information about a particular entry of the state vector which is not local with respect to its neighborhood. Moreover, such a scheme requires the knowledge of the dimension of the state vector at each agent and storage of a high-dimensional local estimates, same as the size of the entire state vector. Such prior knowledge about attributes of the parameter such as dimension in conjunction with requirement for large memory at each agent might be practically infeasible owing to the ad-hoc nature and limited sensing, computation and storage capabilities of agents in a networked setup.\\ Thus, in both of the schemes above, specifically in the case which involves estimating a high-dimensional parameter, it might not be practical to estimate the entire parameter at each agent. In such a high-dimensional parameter estimation scheme, it is highly favorable to estimate only a few entries of the parameter based on the requirements of each agent, which could potentially reduce the dimensions of messages being exchanged in the network thereby reducing the implementation complexity considerably.\\
			\subsection{Connections with Distributed Optimization}
			\label{subsec:dist_opt}
			%Recently, other works in literature have looked at similar setups, albeit in the distributed optimization framework~\cite{mota2015distributed,alghunaim2017distributed}. 
			In principle, distributed stochastic optimization, with each node interested in a few entries of the optimization variable, is more general than the distributed estimation/random fields setup studied here. Indeed, one recovers the setup here with specializing the cost functions to be quadratic. However, this is true only for a very generic formulation of distributed stochastic optimization, where no strong convexity is assumed, each node is interested in a subset of the variable of interest, and the gradient (first order) information is subject to noise, and the underlying network is random. However, to the best of our knowledge, there is no present work that simultaneously addresses all of these aspects. For example, in \cite{mota2015distributed}, the setup involves a static network connected at all times with each agent having access to an incremental first order oracle, i.e., access to exact gradient information; the paper establishes convergence the iterate sequences to the optimizer, however, rates of convergence are not provided. In \cite{alghunaim2017distributed}, the authors consider coupled distributed stochastic optimization setups where the coupling is induced by interest sets of different agents over static networks. The setup in \cite{alghunaim2017distributed} encompasses estimation setups, given that \emph{global observability}\footnote{Global Observability refers to the condition, when the parameter can be reconstructed by stacking the samples collected from all the agents.} holds for each entry of the parameter in the respective clusters, which in turn is subsumed in the setup. Technically speaking, typical distributed optimization setups rely on \emph{local observability}\footnote{Local observability refers to the condition, where an agent can reconstruct its own state based on its own observation sequence.} without assuming \emph{local correctness}\footnote{Local correctness refers to the condition, where the set of local optimizers for the agent's local cost function includes the optimizer of the global objective.} at each agent. However, in the case of distributed estimation, the agents lack \emph{local observability} but preserve \emph{local correctness}. Moreover, the study of the mean square error in \cite{alghunaim2017distributed} reflects errors in terms of the step sizes only and does not reflect explicit dependence in terms of the number of agents collaborating to estimate a particular entry of the parameter. In comparison with \cite{mota2015distributed,alghunaim2017distributed}, we consider a distributed estimation setup over time-varying networks connected only on average and provide asymptotic characterization of the estimator as time goes to $\infty$. Furthermore, we specifically characterize the scaling of the asymptotic variance of each entry of the parameter in terms of the number of agents interested in reconstructing the particular entry in question. We also characterize the fundamental condition so as to generate consistent estimates of each entry of the parameter and show that connectivity of the network and global observability is not enough to ensure consistency of the estimates. We direct the reader to assumption A5 and the discussion after assumption A6 for a detailed illustration. In particular, we establish that connectivity of the subgraphs induced by the interest sets is a sufficient condition to enforce assumption A5. It is an open question as to what is a necessary condition (in terms of the network structure, sensing structure, and the interest sets' structure) so as to enforce assumption A5.
		\section{$\mathcal{CIRFE}$: DISTRIBUTED RANDOM FIELDS ESTIMATION}
		\label{sec:cirfe}
		\noindent In this section, we develop the algorithm $\mathcal{CIRFE}$.
		The parameter to be reconstructed which is the vector of states accumulated over the entire network is $\mathbf{\btheta}^{\ast}\in\mathbb{R}^{N}$. The sparsifying nature of  $\mathbf{H}_{n}$ in \eqref{RLU-prob1} is related to the coupling induced by the measurements in the field. To be specific, let us define $\widetilde{\mathcal{I}}_{n}$ as the set of agents whose states influence the measurement $\mathbf{y}_{n}(t)$ at agent $n$, i.e., $\widetilde{\mathcal{I}}_{n}$ collects the agents for which the corresponding columns of matrix $\mathbf{H}_n$ is non-zero. In what follows, we say an agent $n$ is physically coupled to an agent $l$ if the observation at agent $n$ is influenced by the state component $\theta^{\ast}_{l}$. Typically, $\widetilde{\mathcal{I}}_{n}$ is a small subset of the total number of agents $N$. Technically speaking, the above mentioned coupling induced by the measurements can be expressed in terms of an adjacency matrix, $\hat{\mathbf{A}}$, where $\hat{\mathbf{A}}_{nl}=1$ if $l\in\widetilde{\mathcal{I}}_{n}$ and $0$ otherwise.
		\noindent Now, that we have abstracted out the physical coupling~(physical layer) in the networked system under consideration, we discuss about the communication layer~(cyber layer), i.e., the inter-agent communication network and the associated communication protocol. Before getting into the communication protocol, we introduce \emph{interest sets} of agents' around which the communication protocol is built.
		\noindent We intend to formulate a distributed estimation procedure, where every agent wants to reconstruct the states of a small subset of the agents, which we refer to as the \emph{interest set} of the agent. In what follows, we point out that the $n$-th component of the field has a one-to-one correspondence with the $n$-th agent: this one-to-one correspondence is best illustrated by visualizing the agents to be (geographically) distributed in a field with $\theta^{\ast}_{n}$ representing the state of the field at the location of the $n$-th agent. Formally, the interest set of an agent is represented as $\mathcal{I}_{n}$. The interest set could vary from one agent to another. The interest sets can be arbitrary but need to satisfy the following assumption:
		\vspace{-5pt}
		\begin{myassump}{A3}
			\label{as:3}
			\emph{The set of agents physically coupled with agent $n$ is a subset of the interest set of agent $n$, i.e., $\widetilde{\mathcal{I}}_{n}\subset\mathcal{I}_{n}$}.
		\end{myassump}
		\noindent We assume without loss of generality that $\widetilde{\mathcal{I}}_{n}$ and hence $\mathcal{I}_{n}$ is non-empty for all $n$. (For illustration, see below the example after Assumption \ref{as:6}).
		%\noindent Furthermore, we assume that the interest set of every agent $n$ is non-empty.
		%\noindent The vector of states $\mathbf{\btheta}^{\ast}$ is rendered locally unobservable at all the agents. Hence, the following global observability condition is needed so as to enable each agent to obtain a consistent estimate of its interest set. We formalize the global observability assumption as follows:
		%\begin{myassump}{A4}
		%	\label{as:4}
		%	\emph{The sensing model is globally observable, i.e., any pair $\btheta, \acute{\btheta}$ of possible parameter instances in $\mathbb{R}^{N}$ satisfies $\sum_{n=1}^{N}\left\|\mathbf{f}_{n}(\btheta)-\mathbf{f}_{n}(\acute{\btheta})\right\|^{2}=0$ if and only if $\btheta=\acute{\btheta}$.}
		%\end{myassump}
		%\noindent Global observability corresponds to the centralized setting, where an estimator has access to the agents' observations at all times. The global observability assumption corresponds to the fact that if there was a centralized estimator with simultaneous access to all the agents' measurements, this centralized estimator would be able to reasonably estimate the underlying parameter.
		%\vspace{-15pt}
		\noindent We number the nodes~(equivalently, components of $\btheta$) in the interest sets of agents in increasing order. Thus, the interest set $\mathcal{I}_{n}$ at an agent $n$ can be considered to be a vector with dimension $\left|\mathcal{I}_{n}\right|$. For example, $\mathcal{I}_{n}(r)=p$ indicates that agent $p$ is the $r$-th agent in increasing order in the interest set $\mathcal{I}_{n}$. We also have that $\mathcal{I}_{n}^{-1}(p)=r$. Moreover, as each agent $n$ is only interested in reconstructing the states of agents in its interest set, the estimate at agent $n$, $\mathbf{x}_{n}(t)\in\mathbb{R}^{\left|\mathcal{I}_{n}\right|},~\forall~t$. At every time instant $t$, an agent $n$ simultaneously fuses information received from the neighbors and the latest sensed information to update its parameter estimate. However, as the interest set of agents in the neighborhood might not be the same as that of the agent itself, the information received from the neighbors needs censoring. Let the message received from agent $l$ at time $t$ be denoted by $\mathbf{x}_{l}(t)\in\mathbb{R}^{\left|\mathcal{I}_{l}\right|}$, where $l\in\Omega_{n}$. The censored message processed by agent $n$, $\mathbf{x}_{l,n}^{r}(t)\in\mathbb{R}^{\left|\mathcal{I}_{n}\right|}$ is generated as follows:
		{\small\begin{align}
			\label{eq:cens_1}
			\mathbf{e}_{j}^{\top}\mathbf{x}_{l,n}^{r}(t)=
			\begin{cases}
			\mathbf{e}_{\mathcal{I}_{l}^{-1}\left(\mathcal{I}_{n}(j)\right)}^{\top}\mathbf{x}_{l}(t) & \mathcal{I}_{n}(j)\in\mathcal{I}_{l}\\
			0 & \mbox{otherwise},
			\end{cases}
			\end{align}}
		\noindent where $\mathbf{e}_{j}$ and $\mathbf{e}_{\mathcal{I}_{l}^{-1}\left(\mathcal{I}_{n}(j)\right)}$ are canonical vectors with $\mathbf{e}_{j}\in\mathbb{R}^{\left|\mathcal{I}_{n}\right|}$ and $\mathbf{e}_{\mathcal{I}_{l}^{-1}\left(\mathcal{I}_{n}(j)\right)}\in\mathbb{R}^{\left|\mathcal{I}_{l}\right|}$.
		\noindent Agent $n$ only wants to use estimates of those states from an agent in its neighborhood which are common to their interest sets. Formally, with agent $l$, agent $n$ only wants to use estimates of the states in the set $\mathcal{I}_{n}\cap\mathcal{I}_{l}$. Similarly, while using the obtained estimate states from the neighbors, only those states in the set $\mathcal{I}_{n}\cap\mathcal{I}_{l}$ are updated. We also define the transformed estimate $\mathbf{x}_{l,n}^{s}(t)\in\mathbb{R}^{\left|\mathcal{I}_{n}\right|}$ at agent $n$, for each $l\in\Omega_{n}(t)$ as follows:
		{\small\begin{align}
			\label{eq:cens_2}
			\mathbf{e}_{j}^{\top}\mathbf{x}_{l,n}^{s}(t)=
			\begin{cases}
			\mathbf{e}_{j}^{\top}\mathbf{x}_{n}(t) & \mathcal{I}_{n}(j)\in\mathcal{I}_{l}\\
			0 & \mbox{otherwise}.
			\end{cases}
			\end{align}}
		\noindent where $j\in\{1,\cdots,\left|\mathcal{I}_{n}\right|\}$. The agent $n$ also incorporates the latest sensed information $\mathbf{y}_{n}(t)$ while updating the parameter estimate at each sampling epoch and only retains the components of interest, i.e., those in $\mathcal{I}_n$. For a given vector $\mathbf{z}\in\mathbb{R}^{\left|\mathcal{I}_{n}\right|}$, let $\mathbf{z}^{\mathcal{P}_{\mathcal{I}_n}}\in\mathbb{R}^{N}$ be the vector whose $j$-th component is given by
		{\small\begin{align}
			\label{eq:cens_3}
			\mathbf{e}_{j}^{\top}\mathbf{z}^{\mathcal{P}_{\mathcal{I}_n}}=
			\begin{cases}
			\mathbf{e}_{\mathcal{I}_{n}^{-1}(j)}^{\top}\mathbf{z} & j\in\mathcal{I}_{n}\\
			0 & \mbox{otherwise}.
			\end{cases}
			\end{align}}
		\noindent Finally, for a given vector $\mathbf{z}\in\mathbb{R}^{N}$, $\mathbf{z}_{\mathcal{I}_n}$ denotes the vector in $\mathbb{R}^{\left|\mathcal{I}_{n}\right|}$, where $\mathbf{e}_{j}^{\top}\mathbf{z}_{\mathcal{I}_n}=\mathbf{e}_{\mathcal{I}_{n}(j)}^{\top}\mathbf{z}$.

		\noindent We now introduce the algorithm $\mathcal{CIRFE}$ for distributed parameter estimation:
		{\small\begin{align}
			\label{eq:cirfe_1}
			&\mathbf{x}_{n}(t+1)=\mathbf{x}_{n}(t)-\underbrace{\sum_{l\in\Omega_{n}(t)}\beta_{t}\left(\mathbf{x}_{l,n}^{s}(t)-\mathbf{x}_{l,n}^{r}(t)\right)}_{\text{Neighborhood Consensus}}\nonumber\\&+\underbrace{\alpha_{t}\mathbf{H}_{n}^{\top}\mathbf{R}_{n}^{-1}\left(\mathbf{y}_{n}(t)-\mathbf{H}_{n}\mathbf{x}_{n}^{\mathcal{P}_{\mathcal{I}_n}}(t)\right)_{\mathcal{I}_{n}}}_{\text{Local Innovation}},
			\end{align}}
		\noindent where $\Omega_{n}(t)$ represents the neighborhood of agent $n$ at time $t$; and $\left\{\beta_{t}\right\}$ and $\left\{\alpha_{t}\right\}$ are the consensus and innovation weight sequences given by
		{\small\begin{align}
			\label{eq:ab}
			\beta_{t} = \frac{\beta_{0}}{(t+1)^{\delta_{1}}}, \alpha_{t}=\frac{a}{t+1},
			\end{align}}
		\noindent where $a, b > 0$ and $0<\delta_{1}<1/2-1/(2+\epsilon_{1})$ and $\epsilon_{1}$ was as defined in Assumption \ref{as:1}. It is to be noted that with the interest set of each agent being $\mathcal{I}_{n}=\{1,2,\cdots,N\}$, we have that the update in \eqref{eq:cirfe_1} reduces to the classical consensus+innovations update for linear parameter estimation schemes~(see, \cite{kar2013distributed} for example). Thus, the classical consensus+innovations parameter estimation scheme, is strictly a special case of the update in \eqref{eq:cirfe_1}. 
		
		%*****************
			\noindent We now illustrate the introduced setup and algorithm \eqref{eq:cirfe_1} with a $5$ agents network example in Fig.~\ref{5node-RLU}.
			\begin{figure}%[ptb]
				\begin{center}
					\includegraphics[height=1in, width=1in ] {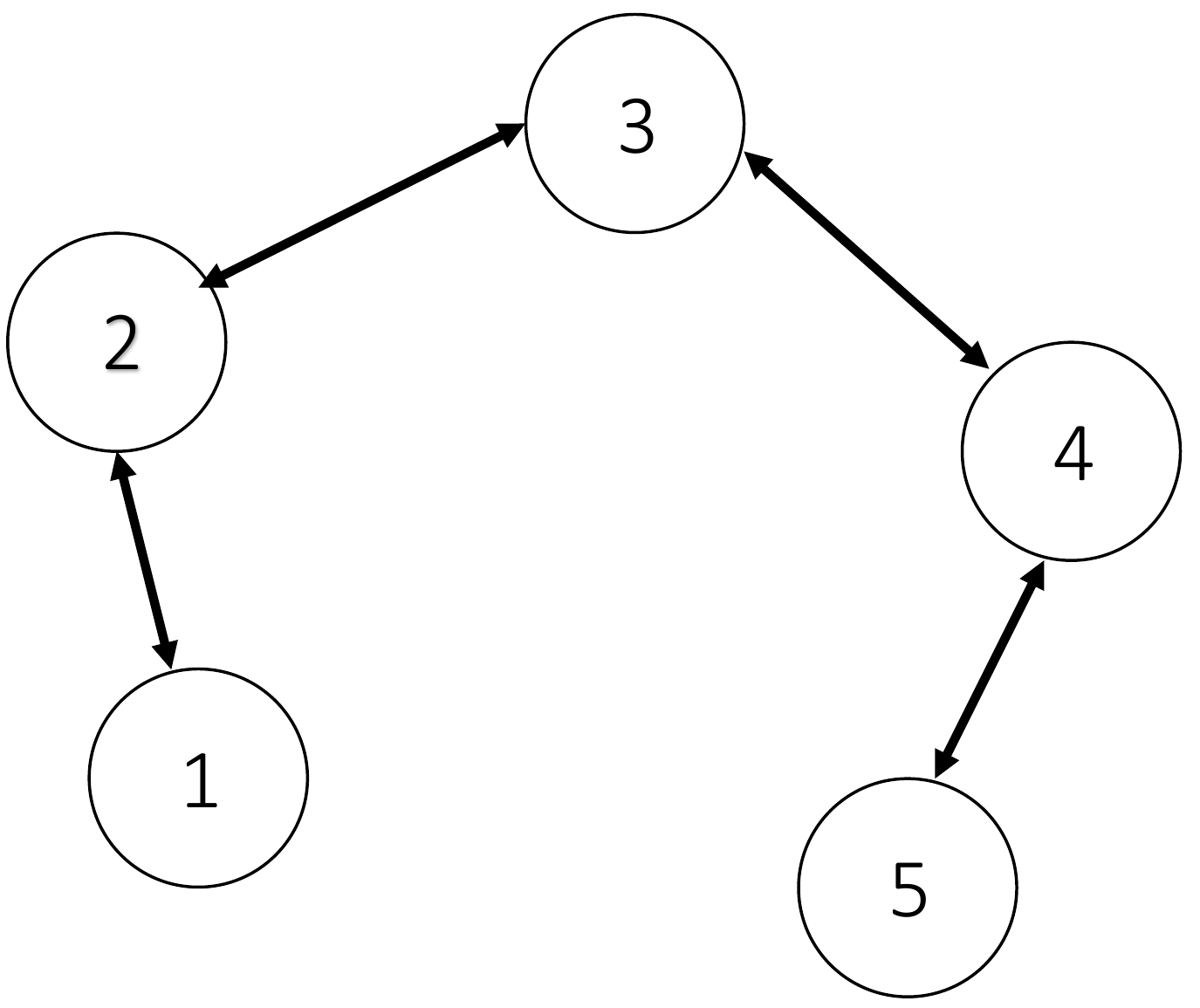} \caption{A network example emphasizing the notion of structural observability.}
					\label{5node-RLU}
				\end{center}
			\end{figure}
			\noindent Each node $n$ corresponds to a physical component $\theta_{n}^{\ast}$. Thus, $\btheta^{\ast}\in\mathbb{R}^{5}$. The solid lines connecting the nodes correspond to the inter-node communication pattern. Each node observes a noisy scalar functional. In particular, we assume
			{\small\begin{align}
				\label{RLU-dis2}
				&y_{3}(t)=\frac{1}{3}\left(\theta_{2}^{\ast}+\theta_{3}^{\ast}+\theta_{4}^{\ast}\right)+\gamma_{3}(t)\nonumber\\
				&y_{n}(t)=\theta_{n}^{\ast}+\gamma_{n}(t), n=1,2,4,5.
				\end{align}}
			\noindent Note that then the noise covariance matrix 
			$\mathbf{R}_n$ is a positive scalar, $n=1,2,...,5$. 
			Also, for $n \neq 5$, $\mathbf{H}_n$ is a 5-dimensional (row)  
			vector with all entries equal to zero except the $n$-th entry 
			which equals one. On the other hand, 
			$\mathbf{H}_3 = [\,0,\,1/3,\,1/3,\,1/3,\,0\,]$.  
			we have that $\widetilde{\mathcal{I}_{n}}= \{n\}$ for $n=1,2,4,5$, and $\widetilde{\mathcal{I}_{3}}= \{2,3,4\}$.
			Let us also assume that the agents' interest sets are
			given by ${\mathcal{I}_{n}} = \widetilde{\mathcal{I}_{n}}$, for each $n=1,2,...,5$.
			For notational simplicity, we omit time index~$t$ when writing the
			agents' estimates; that is, we write $\mathbf{x}_n$
			in place of $\mathbf{x}_n(t)$.
			Also, we denote by $[\mathbf{x}_n]_i$
			the $i$-th entry of $\mathbf{x}_n$.
			Then, agent 3's
			estimate $\mathbf{x}_3$
			is a $3 \times 1$ vector,
			with $[\mathbf{x}_3]_1$ being an estimate of
			$\theta_{2}^{\ast}$,
			$[\mathbf{x}_3]_2$ being an estimate of
			$\theta_{3}^{\ast}$,
			and
			$[\mathbf{x}_3]_3$ being an estimate of
			$\theta_{4}^{\ast}$.
			Regarding the remaining agents~$n \neq 3$, we have that
			$\mathbf{x}_n$
			is a scalar, with $\mathbf{x}_n$ being an estimate of $\theta_{n}^{\ast}$.
			Next, consider agent~$3$ and its interaction with agent 2.
			The censored quantity
			$\mathbf{x}_{32}^r$
			at agent 3 based on the received
			message from agent~2 equals
			$\mathbf{x}_{32}^r = [\,\mathbf{x}_2,\,0,\,0\,]^\top$.
			Further, the agent 3's own censored estimate, adapted
			so that it can be combined with $\mathbf{x}_{32}^r$, equals
			$\mathbf{x}_{32}^s = [\,[\mathbf{x}_3]_1,\,0,\,0\,]^\top$.
			Note that the first entry in both  $\mathbf{x}_{32}^r$ and
			$\mathbf{x}_{32}^s$ corresponds to an
			estimate of $\theta_{2}^{\ast}$, the second entry of
			both $\mathbf{x}_{32}^r$ and
			$\mathbf{x}_{32}^s$ corresponds to an
			estimate of $\theta_{3}^{\ast}$,
			and the third entry of
			both $\mathbf{x}_{32}^r$ and
			$\mathbf{x}_{32}^s$ corresponds to an
			estimate of $\theta_{4}^{\ast}$.
			The second and third entry in both $\mathbf{x}_{32}^r$ and
			$\mathbf{x}_{32}^s$ is zero,
			because the intersection of the agents' 2 and 3 interest sets
			$\mathcal{I}_1 \cap \mathcal{I}_2 = \{2\}$, i.e.,
			it does not include the interest for $\theta_{3}^{\ast}$ nor for $\theta_{4}^{\ast}$.
			Further, we have that $\mathbf{x}_{23}^r = [\mathbf{x}_3]_1$
			and $\mathbf{x}_{23}^s = \mathbf{x}_2$.
			The remaining pairs of quantities
			$\mathbf{x}_{nl}^r $ and
			and $\mathbf{x}_{nl}^s $ are defined analogously.
			Next,
			agent 3's estimate
			``lifted'' to the $N=5$-dimensional space
			equals $\widetilde{\mathbf{x}}_3 = [0,\,\,[\mathbf{x}_3]_1,\,[\mathbf{x}_3]_2,\,[\mathbf{x}_3]_3,\,0\,]^\top$.
			Note that the first and fifth entries in
			$\widetilde{\mathbf{x}}_3$ are zero, because
			agent 3 does not have interest in $\theta_{1}^{\ast}$ nor in $\theta_{5}^{\ast}$.
			Similarly, we have that
			$\widetilde{x}_2 = [\,0,\,\mathbf{x}_2,\,0,\,0,\,0\,]^\top$.
			We next specialize the update rule \eqref{eq:cirfe_1} for the
			example considered here and agent 3; we have:
			{\small\begin{align}
				\label{eq:cirfe_1_example}
				&\underbrace{\begin{bmatrix}
					[\mathbf{x}_{3}]_1(t+1) \\
					[\mathbf{x}_{3}]_2(t+1) \\
					[\mathbf{x}_{3}]_3(t+1)
					\end{bmatrix}}_{\mathbf{x}_3(t+1)}
				=\underbrace{\begin{bmatrix}
					[\mathbf{x}_{3}]_1(t) \\
					[\mathbf{x}_{3}]_2(t) \\
					[\mathbf{x}_{3}]_3(t)
					\end{bmatrix}}_{\mathbf{x}_3(t)}+\beta_{t}
				\underbrace{\left(
					\begin{bmatrix}
					\mathbf{x}_{2}(t) \\
					0 \\
					0
					\end{bmatrix}-
					\begin{bmatrix}
					[\mathbf{x}_{3}]_1(t) \\
					0 \\
					0
					\end{bmatrix}
					\right)}_{\mathbf{x}_{32}^r(t) - \mathbf{x}_{32}^s(t)  }\nonumber\\
				&
				+\beta_{t}\underbrace{\left(
					\begin{bmatrix}
					0 \\
					0 \\
					\mathbf{x}_{4}(t)
					\end{bmatrix}-
					\begin{bmatrix}
					0 \\
					0 \\
					[\mathbf{x}_{3}]_3(t)
					\end{bmatrix}
					\right)}_{\mathbf{x}_{34}^r(t) - \mathbf{x}_{34}^s(t)  }\nonumber\\&+
				\alpha_t\,
				\underbrace{
					\begin{bmatrix}
					1/3 \\
					1/3 \\
					1/3
					\end{bmatrix}}_{({\mathbf{H}_3^\top})_{\mathcal{I}_{3}}}\,R_3^{-1}\left( y_3(t)-\frac{1}{3}([\mathbf{x}_3]_1(t) +
				[\mathbf{x}_3]_2(t)
				+
				[\mathbf{x}_3]_3(t) )\right).
				\end{align}}
		%***********************
		We formalize an assumption on the connectivity of the inter-agent communication graph before proceeding further.
		\begin{myassump}{A4}
			\label{as:4}
			\emph{The inter-agent communication graph is connected on average, i.e., $\lambda_{2}(\mathbf{\overline{L}}) > 0$, where $\mathbf{\overline{L}}$ denotes the mean of the sequence of identically and independently distributed~(i.i.d) graph Laplacian sequence $\left\{\mathbf{L}(t)\right\}$}.
		\end{myassump}
		% It is to be noted that a pair of nodes $n, l$ can exchange information in its neighborhood at time $t$, when $\psi_{n,t}\neq 0$. At the same time, when $\psi_{n,t}= 0$, node $n$ neither transmits nor receives information. The link between node $n$ and node $l$ gets assigned a weight of $\rho_{t}^{2}$ if and only if $\psi_{n,t}\neq 0$ and $\psi_{l,t}\neq 0$.
		%The communication cost per node for the proposed algorithm is given by $\mathcal{C}_{t} = \sum_{s=0}^{t-1}\zeta_{s} = \Theta\left(t^{1+(\epsilon-\tau_{1})/2}\right)$,
		%which in turn is strictly sub-linear as $\epsilon < \tau_{1}$.
		\begin{Remark}
			\label{rm:1}
			In the parameter estimation scheme in \eqref{eq:cirfe_1}, an agent $n$ uses only those components of its neighbor $l$'s estimate $\mathbf{x}_{l}(t)$, which belong to its interest set $\mathcal{I}_{n}$. Thus, agents $n$ and $l$ combine components linearly which belong to $\mathcal{I}_{n}\cap\mathcal{I}_{l}$ and reject the rest of the components. From an implementation viewpoint, it is desirable for an agent $l$ to only transmit those components to agent $n$ which belong to $\mathcal{I}_{n}\cap\mathcal{I}_{l}$ instead of transmitting the entire $\mathbf{x}_{l}(t)$ to agent $n$ as the one which involves exchanging only those components which are common to the agents has lower communication overhead. In the former case, the receiving agent $n$ will zero out the components it does not require, so both the transmission strategies would lead to the same update. Moreover, in the innovation term, where an agent $n$ uses its own previous state to compute the innovation, an agent subsequently retains only the components of interest so as to keep the update economical in terms of size. We also emphasize here that the inter-agent communication graphs $\{L(t)\}$ and the physical adjacency matrix $\hat{A}$ induced by the measurement coupling may be structurally different.
		\end{Remark}
		\noindent We now present a more compact representation of the $\mathcal{CIRFE}$ algorithm so as to be able to establish its asymptotic convergence properties.
		\noindent Let $\mathcal{I}$ denote a subset of $\left\{1, 2, \cdots, N\right\}$. Define the diagonal matrix $P_{\mathcal{I}}$ which selects the corresponding non-zero components of $\mathcal{I}$ from a $\mathbb{R}^{N^2}$ dimensional vector. In particular, $P_{\mathcal{I}} = \textrm{diag}\left[P_{\mathcal{I}_{1}},\cdots,P_{\mathcal{I}_{n}}\right]$, where each $P_{\mathcal{I}_{n}}\in\mathbb{R}^{N\times N}$ and is a diagonal matrix such $[P_{\mathcal{I}_{n}}]_{i,i}=1$ if $i\in\mathcal{I}_{n}$ or $0$ otherwise. 
		
		%**************
			\noindent For the 5-agent network example associated with Figure~1, 
			we have for $n \neq 3$ that $P_{\mathcal{I}_{n}}$ 
			is the $5 \times 5$ matrix with all 
			the entries equal to zero, except the 
			$(n,n)$-th entry which equals one.  
			The matrix $P_{\mathcal{I}_{3}}$ 
			has all the entries equal to zero, 
			except the $(2,2)$-th, $(3,3)$-th, and $(4,4)$-th entries, 
			which al equal to one.
		
		%**************
		
		\noindent For the estimate sequence $\left\{\mathbf{x}_{n}(t)\right\}$ at agent $n$, let $\left\{\widetilde{\mathbf{x}}_{n}(t)\right\}\in\mathbb{R}^{N}$ denote the auxiliary estimate sequence, where $\widetilde{\mathbf{x}}_{n}(t)=\mathbf{x}_{n}(t)^{\mathcal{P}_{\mathcal{I}_n}}$. With the above development in place, it is easy to see that, for $\mathbf{y}\in\mathbb{R}^{N}$,  $\left(\mathbf{x}_{l,n}^{r}(t)\right)^{\mathcal{P}_{\mathcal{I}_n}}=P_{\mathcal{I}_{n}}P_{\mathcal{I}_{l}}\widetilde{\mathbf{x}}_{l}(t)$, $\left(\mathbf{x}_{l,n}^{s}(t)\right)^{\mathcal{P}_{\mathcal{I}_n}}=P_{\mathcal{I}_{n}}P_{\mathcal{I}_{l}}\widetilde{\mathbf{x}}_{n}(t)$ and $\left(\mathbf{x}_{n}(t)\right)^{\mathcal{P}_{\mathcal{I}_n}}=P_{\mathcal{I}_{n}}\widetilde{\mathbf{x}}_{n}(t)$. The $\mathcal{CIRFE}$ update in \eqref{eq:cirfe_1} can then be written in terms of the auxiliary processes as follows:
		{\small\begin{align}
			\label{eq:cirfe_2}
			&\widetilde{\mathbf{x}}_{n}(t+1)=\widetilde{\mathbf{x}}_{n}(t)-\sum_{l\in\Omega_{n}(t)}\beta_{t}P_{\mathcal{I}_{n}}P_{\mathcal{I}_{l}}\left(\widetilde{\mathbf{x}}_{n}(t)-\widetilde{\mathbf{x}}_{l}(t)\right)\nonumber\\&+\alpha_{t}P_{\mathcal{I}_{n}}\mathbf{H}_{n}^{\top}\mathbf{R}_{n}^{-1}\left(\mathbf{y}_{n}(t)-\mathbf{H}_{n}P_{\mathcal{I}_{n}}\widetilde{\mathbf{x}}_{n}(t)\right).
			\end{align}}
		\noindent We introduce the matrix $\mathbf{L}_{\mathcal{P}}\in\mathbb{R}^{N^{2}\times N^{2}}$ so as to make the above representation more compact.
		{\small\begin{align}
			\label{eq:LP}
			\left[\mathbf{L}_{\mathcal{P}}(t)\right]_{nl}=
			\begin{cases}
			-P_{\mathcal{I}_{n}}\sum_{r=1:r\neq n}^{N}\mathbf{L}_{nr}(t)P_{\mathcal{I}_{r}}& \mbox{if}~n=l\\
			\mathbf{L}_{nl}(t)P_{\mathcal{I}_{l}}P_{\mathcal{I}_{n}} & \mbox{otherwise},
			\end{cases}
			\end{align}}
		\noindent where $\left[\mathbf{L}_{\mathcal{P}}(t)\right]_{nl}\in\mathbb{R}^{N\times N}$ denotes the $(n,l)$-th sub-block of the block matrix $\mathbf{L}_{\mathcal{P}}$. It follows by elementary matrix multiplication properties that $\mathcal{P}\mathbf{L}_{\mathcal{P}}(t) = \mathbf{L}_{\mathcal{P}}(t)$. It is also to be noted that $\mathbf{L}_{\mathcal{P}}$ is a symmetric matrix.
		The matrix $\mathbf{L}_{\mathcal{P}}(t)$ at each time step $t$ can be decomposed as follows:
		{\small\begin{align}
			\label{eq:LP_1}
			\mathbf{L}_{\mathcal{P}}(t) = \overline{\mathbf{L}_{\mathcal{P}}}+\widetilde{\mathbf{L}_{\mathcal{P}}}(t),
			\end{align}}
		\noindent where $\{\mathbf{L}_{\mathcal{P}}(t)\}$ is an i.i.d. sequence with mean $\overline{\mathbf{L}_{\mathcal{P}}}$ and $\widetilde{\mathbf{L}_{\mathcal{P}}}(t) = \mathbf{L}_{\mathcal{P}}(t) - \mathbb{E}\left[\mathbf{L}_{\mathcal{P}}(t)\right]$. Thus, we have that the residual sequence $\{\widetilde{\mathbf{L}_{\mathcal{P}}}(t)\}$ satisfies $\mathbb{E}\left[\widetilde{\mathbf{L}_{\mathcal{P}}}(t)\right]=\mathbf{0}$.
		%where,
		%\begin{align}
		%\label{eq:laplace_res1}
		%\mathbb{E}\left[\left\|\widetilde{\mathbf{L}_{\mathcal{P}}}(t)\right\|^{2}\right] \leq  \frac{c_{r}\beta_{0}\rho_{0}^{2}}{(t+1)^{\tau_{1}+\epsilon}} .
		%\end{align}
		
		\noindent With the above development in place, the update in \eqref{eq:cirfe_2} can be written in a compact form as follows:
		{\small\begin{align}
			\label{eq:cirfe_3}
			\widetilde{\mathbf{x}}(t+1)=\widetilde{\mathbf{x}}(t)-\beta_{t}\mathbf{L}_{\mathcal{P}}(t)\widetilde{\mathbf{x}}(t)+\alpha_{t}\mathcal{P}\mathbf{G}_{H}\mathbf{R}^{-1}\left(\mathbf{y}(t)-\mathbf{G}_{H}^{\top}\mathcal{P}\widetilde{\mathbf{x}}(t)\right),
			\end{align}}
		where $\widetilde{\mathbf{x}}^{\top}(t)=\left[\widetilde{\mathbf{x}}^{\top}_{1}(t),\cdots,\widetilde{\mathbf{x}}^{\top}_{N}(t)\right]^{\top}$, $\mathbf{y}(t)^{\top}=[y_{1}(t)^{\top}\cdots y_{N}(t)^{\top}]^{\top}$, $\mathbf{R}=diag\left[\mathbf{R}_{1},\cdots,\mathbf{R}_{N}\right]$, $\mathcal{P}=diag\left[\mathcal{P}_{\mathcal{I}_1},\cdots,\mathcal{P}_{\mathcal{I}_N}\right]$, and $\mathbf{G}_{H}=diag[\mathbf{H}_{1}^{\top}, \mathbf{H}_{2}^{\top},\cdots, \mathbf{H}_{N}^{\top}]$.\\
		\begin{Remark}
				\label{rm:1.1}
				In the case when the noise covariance is not known apriori, a recursive estimator of the inverse noise covariance can be used so as to be used as a plugin estimate for $\mathbf{R}_{n}^{-1}$. A plugin estimate for $\mathbf{R}_{n}^{-1}$ at time $t+1$, denoted by $\widehat{\mathbf{R}}_{n}^{-1}(t+1)$ can be generated as follows:
				{\small\begin{align*}
					&\mathbf{Q}_{n}(t+1) = \frac{1}{t}\sum_{s=0}^{t}\mathbf{y}_{n}(s)\mathbf{y}_{n}^{\top}(s)-\left(\frac{1}{t}\sum_{s=0}^{t}\mathbf{y}_{n}(s)\right)\left(\frac{1}{t}\sum_{s=0}^{t}\mathbf{y}_{n}(s)\right)^{\top}\nonumber\\
					&\widehat{\mathbf{R}}_{n}^{-1}(t+1) = \left(\mathbf{Q}_{n}(t+1)+\gamma_{t}\mathbf{I}_{M_{n}}\right)^{-1},
					\end{align*}}
				where $\gamma_{t}$ is a time-decaying sequence such that $\gamma_{t}\to 0$ as $t\to\infty$.
				
				\noindent Also, given the sensing model and the assumption that the dimension of the observations at each agent $n$, given by $M_{n}$ is $M_{n}\ll N$, inverting a low-dimensional matrix is not particularly computationally taxing. In particular, $M_{n}$ can be equal to $1$ for instance in which the inverse noise covariance matrix can be estimated seamlessly. Furthermore, it is to be noted that the update can be adapted to be of the following form, where $\mathbf{R}^{-1}$ is replaced by $\mathbf{I}$
				\begin{align*}
				\widetilde{\mathbf{x}}(t+1)=\widetilde{\mathbf{x}}(t)-\beta_{t}\mathbf{L}_{\mathcal{P}}(t)\widetilde{\mathbf{x}}(t)+\alpha_{t}\mathcal{P}\mathbf{G}_{H}\left(\mathbf{y}(t)-\mathbf{G}_{H}^{\top}\mathcal{P}\widetilde{\mathbf{x}}(t)\right),
				\end{align*}
				which does not require the inverse noise covariance. We remark that with the above update, the algorithm still retains the property concerning the almost sure convergence of the parameter estimate at each agent to the entries of the parameter corresponding to its interest set. Thus, the computational cost can be reduced drastically with an update of the following form as defined above, which does not involve any matrix inversions. Thus, when knowledge or calculation of $\mathbf{R}^{-1}$ is an issue, algorithm \eqref{eq:cirfe_3} can be replaced with the update above, retaining consistency but possibly with a loss in terms of the asymptotic covariance.
			\end{Remark}
			\begin{Remark}
				\label{rm:2}
				The recursive update in \eqref{eq:cirfe_3} is of the stochastic approximation type. The stochastic approximation procedure, employed here is a mixed time-scale stochastic approximation as opposed to the classical single time-scale stochastic approximation~(see, for example \hspace{-2pt}\cite{nevelson1973stochastic}). The above notion of mixed time-scale is very different from the more commonly studied two time-scale stochastic approximation (see, for instance \cite{Borkar-stochapp}) in which a fast process is coupled with a slower dynamical system. The approach employed here is similar to the ones in \cite{Gelfand-Mitter} and \cite{kar2013distributed} in which a single update procedure is influenced by multiple potentials with different time-decaying weights. Now, suppose that the interest set of each agent consists of all components of $\btheta^{\ast}$, i.e., the update in \eqref{eq:cirfe_3} reduces to the classical consensus+innovations update in (2). A key technical step employed in the analysis of classical consensus+innovations procedures of the type in \eqref{eq:benchmark_ci} (see, for example, \cite{kar2013distributed}) consists of an approximation of the update in \eqref{eq:benchmark_ci} to a single time-scale stochastic approximation procedure that is \emph{asymptotically equivalent} to the former, in particular, that converges to the original iterate sequence at a rate faster than $(t+1)^{0.5}$. Typically, in the context of \eqref{eq:benchmark_ci} the approximating single time-scale procedure is the network-averaged estimate sequence, $\widetilde{\mathbf{x}}_{avg}(t) = \left(\frac{\mathbf{1}_{N}^{\top}}{N}\otimes\mathbf{I}_{N}\right)\widetilde{\mathbf{x}}(t)$, and the analysis in \cite{kar2013distributed} uses the fact that the Laplacian $\mathbf{L}(t)$ in \eqref{eq:benchmark_ci} has a left eigen vector of $\mathbf{1}_{N^{2}}$ and that every agent is interested in estimating the entire parameter vector. However, in the context of the update in \eqref{eq:cirfe_3}, every agent is interested in only a few entries of the parameter which makes the characterization of asymptotic properties of the estimate sequences highly non-trivial and substantially different from prior work on consensus+innovations type estimation procedures~\cite{kar2013distributed} in which agents share the common objective of estimating all components of the parameter. However, in contrast to prior work on consensus+innovations type estimation procedures~(see, for example \cite{kar2013distributed}) in which agents share the common objective of estimating all components of the parameter, the analysis with heterogeneous agent objectives in \eqref{eq:cirfe_3}, in that each agent is interested in a different subset of components, requires new technical machinery. In particular, to obtain asymptotic properties of \eqref{eq:cirfe_3}, we develop a more generalized approximation of the mixed time-scale procedure to an appropriate single time-scale procedure that takes into account of the heterogeneity in agent objectives; this approximation and subsequent analysis require new technical tools that we develop in this paper.
			\end{Remark}
			\noindent Define the subspace $\mathcal{S}_{P}\in\mathbb{R}^{N^2}$ by $\mathcal{S}_{P}=\left\{\mathbf{y}\in\mathbb{R}^{N^2}|\mathbf{y}=\mathcal{P}\mathbf{w}, \mbox{for some}~\mathbf{w}\in\mathbb{R}^{N^{2}}\right\}$. We now formalize a key assumption relating the interest sets $\mathcal{I}_n$ to the network connectivity and global observability.
			\begin{myassump}{A5}
				\label{as:5}
				\emph{There exists a constant $c_{1}>0$ such that,
					{\small\begin{align}
						\label{eq:as4}
						&\mathbf{y}^{\top}\left(\frac{\beta_0}{\alpha_0}\overline{\mathbf{L}_{\mathcal{P}}}+\mathcal{P}\mathbf{G}_{H}\mathbf{R}^{-1}\mathbf{G}_{H}^{\top}\mathcal{P}\right)\mathbf{y}\nonumber\\&\geq c_{1}\left\|\mathbf{y}\right\|^{2}, \forall~\mathbf{y}\in\mathcal{S}_{P}.
						\end{align}}
				}
			\end{myassump}
			\noindent We formalize an assumption on the innovation gain sequence $\{\alpha_{t}\}$ before proceeding further.
			\begin{myassump}{A6}
				\label{as:6}
				\emph{Let $\lambda_{min}\left(\cdot\right)$ denote the smallest eigenvalue. We require that $a$ satisfies,
					{\small\begin{align*}
						a\min\{\lambda_{min}\left(\sum_{n=1}^{N}\mathcal{P}_{\mathcal{I}_n}\mathbf{H}_{n}^{\top}\mathbf{R}_{n}^{-1}\mathbf{H}_{n}\mathcal{P}_{\mathcal{I}_n}\right),c_{1},\beta_{0}^{-1}\}\ge 1,
						\end{align*}}
					where $c_1$ is defined in \eqref{eq:as4}.}
			\end{myassump}
			\noindent  It is to be noted that in Assumption \ref{as:5}, if $\mathcal{P}=\mathbf{I}_{N^2}$, then the subspace $\mathcal{S}_{P}$ reduces to $\mathbb{R}^{N^{2}}$ and the condition in \eqref{eq:as4} reduces to a commonly employed Lyapunov condition in classical consensus+innovations type inference procedures (see, for example, Lemma 6 in \cite{kar2011convergence}) which, in turn, can be enforced by global observability and the mean connectivity of the network under consideration. However, in the case when $\mathcal{P}\neq\mathbf{I}_{N^2}$, the case considered in this paper, global observability and connectivity of the network is not sufficient to obtain the condition in \eqref{eq:as4}. The insufficiency of global observability and connectivity of the network in order to enforce \eqref{eq:as4} can be attributed to heterogeneous objectives of the agents and censoring of messages at agents leading to an inherent information loss. Intuitively, such a condition calls for existence of information pathways between agents who share a particular component in their interest sets and the particular component in question to be observable at this set of agents collectively. As we show in the following~(Lemma \ref{str_RLU}), a sufficient condition for Assumption \ref{as:5} is that in addition to the global observability and the mean network connectedness, the induced subgraph for every entry of the vector $\mathbf{\btheta}^{\ast}$ needs to be connected. The induced subgraph for the $r$-th entry is the set of agents and their associated links which have the $r$-th entry of $\mathbf{\btheta}^{\ast}$ in their interest sets.
			
			\noindent In the following, we will establish consistency of the $\mathcal{CIRFE}$ under Assumption \ref{as:5}. We now show by a simple example that, in general, Assumption \ref{as:5} is stronger than mean connectivity and global observability. To this end, consider again the simple network consisting of $5$ nodes in Fig.~\ref{5node-RLU} and \eqref{RLU-dis2}.
			%\begin{figure}%[ptb]
			%	\begin{center}
			%		\includegraphics[height=1in, width=1in ] {net_example.pdf} \caption{A network example emphasizing the notion of structural observability.}
			%		\label{5node-RLU}
			%	\end{center}
			%\end{figure}

			%********
			\noindent %Each node $n$ corresponds to a physical component $\theta_{n}^{\ast}$. Thus, $\btheta^{\ast}\in\mathbb{R}^{5}$. The solid lines connecting the nodes correspond to the inter-node communication pattern. Each node observes a noisy scalar functional. In particular, we assume
			%{\small\begin{equation}
			%\label{RLU-dis2}
			%y_{3}(i)=\frac{1}{3}\left(\theta_{2}^{\ast}+\theta_{3}^{\ast}+\theta_{4}^{\ast}\right)+\gamma_{3}(i),
			%\end{equation}}
			%{\small\begin{equation}
			%\label{RLU-dis3}
			%y_{n}(i)=\theta_{n}^{\ast}+\gamma_{n}(i), n=1,2,4,5.
			%\end{equation}}
			%%\begin{equation}
			%%\label{RLU-dis4}
			%%z_{5}(i)=\frac{1}{2}\left(\theta_{5}+\theta_{4}\right)+\zeta_{5}(i)
			%%\end{equation}
			%\noindent 
			Clearly, in this case, $G=\sum_{n=1}^{5}\mathbf{H}_{n}^{\top}\mathbf{H}_{n}$ is invertible and, as shown, the communication network is connected. %Also, $\widetilde{\mathcal{I}_{n}}= \{n\}$ for $n=1,2,4,5$ and $\widetilde{\mathcal{I}_{3}}= \{2,3,4\}$.
			\noindent In case, every node wants to estimate the entire $\btheta^{\ast}$, then the above inference task reduces to the inference setup considered in \cite{kar2011convergence,sahu2016distributedtsipn}.
			\noindent Consider the case where $\mathcal{I}_{n}=\widetilde{\mathcal{I}}_{n}$ for $n=1,2,3,4$, i.e., these nodes are interested in reconstructing only their own states and those who influence their observations. However, let $\mathcal{I}_{5}=\{5,1\}$, i.e., node $5$ is interested in the state of node $1$. This problem falls under the purview of $\mathcal{CIRFE}$. Clearly, Assumption \ref{as:4} is satisfied. However, it can be shown by calculating the various terms, that assumption \ref{as:5} is not satisfied and hence, convergence of $\mathcal{CIRFE}$ to desired values is not guaranteed. This shows that mean connectivity and global observability is not sufficient for assumption~\ref{as:5} in general. We provide an intuitive explanation, why the $\mathcal{CIRFE}$ is not expected to yield accurate estimates in this case and why the Lyapunov type requirement in assumption \ref{as:5} is sufficient for $\mathcal{CIRFE}$'s desired convergence.. Looking at Fig.~\ref{5node-RLU}, we note that the only node that observes (at least partially) the component $\theta_{1}^{\ast}$ is node $1$, i.e., the influence of the state $\theta_{1}^{\ast}$ only affects the observations at node $1$. Clearly, for node $5$ to be able to reconstruct $\theta_{1}^{\ast}$, it should be able to access information about $\theta_{1}^{\ast}$ from the allowed communication graph. Moreover, there is a path connecting node $1$ to node $5$. However, the other nodes in the path are not interested in reconstructing $\theta_{1}^{\ast}$, so they do not participate in the exchange of information regarding $\theta_{1}^{\ast}$. For example, node $2$ ignores the estimate of $\theta_{1}^{\ast}$ at node $1$ and similarly the others. As a result, the information about $\theta_{1}^{\ast}$ never reaches node $5$, although the communication network is connected. Note that the induced subgraph of component $1$ of $\btheta^{*}$ is disconnected, and it involves only nodes $1$ and $5$ and no links.
			
			\noindent At the same time, it is easy to see that this problem is resolved if an extra communication link is added between nodes $1$ and $5$. Thus, we see that connectivity of the subgraph formed by those nodes interested in reconstructing $\theta_{1}$ seems to facilitate proper information flow necessary for the desired convergence of $\mathcal{CIRFE}$. Based on this intuition, we formulate a general structural connectivity condition~(see \cite{kar2010large}) that guarantees the satisfaction of \ref{as:5} which, in turn, will be used subsequently to derive the convergence of $\mathcal{CIRFE}$. We direct the reader to Lemma 3.4.1 in \cite{kar2010large} for a proof.
			
			\begin{Lemma}[Lemma 3.4.1 in \cite{kar2010large}]
				\label{str_RLU}
				Let assumption \ref{as:4} be satisfied and the global observability condition hold. For each component $r$ of $\mathbf{\theta}^{\ast}$, define the subset $\mathcal{I}^{r}\subset [1,\cdots,N]$ by
				{\small\begin{equation}
					\label{str_RLU1}
					\mathcal{I}^{r}=\{n\in [1,\cdots,N]~|~r\in\mathcal{I}_{n}\}
					\end{equation}}
				\noindent Let $\overline{\mathcal{G}}$ denote the network graph corresponding to the mean Laplacian $\overline{\mathbf{L}}$, i.e., there is an edge between
				nodes $n$ and $l$ in $\overline{G}$ \emph{iff} the $(n,l)$-th entry in $\overline{\mathbf{L}}$ is non-zero. For each $1\leq r\leq N$, denote the induced subgraph $\overline{\mathcal{G}}_{r}$ of $\overline{\mathcal{G}}$ with node set $\mathcal{I}^{r}$. Then, condition~\ref{as:5} is satisfied if $\overline{\mathcal{G}}_{r}$ is connected for all $r$.
			\end{Lemma}
			\noindent Technically speaking, the average connectedness of the induced subgraphs in conjunction with the global observability of the entry of the parameter relevant to the subgraphs is enough to ensure consistency of the estimate sequence of the entry of the parameter. The combinatorial perspective brought about by the preceding observation being, can one relax the connectivity of the induced subgraph. For example, consider the $r$-th entry of the parameter. Let the number of agents interested to estimate the entry is $N_{r}$ out of which $M$ agents (referred to as $\mathcal{O}$-agents) have the entry incorporated into their observations. In the case, when one can split $N_r$ agents into disconnected components where each component consists of non-zero number of agents which observe the entry and the entry is rendered globally observable with respect to those $\mathcal{O}$-agents in that component, would ensure the estimates of that entry being consistent at each agent which is interested to reconstruct that agent. However, as the subgraphs induced by interest sets are coupled in lieu of the interest sets, it might not be possible to ensure such a construction as the one described before for each entry of the parameter.

			\section{$\mathcal{CIRFE}$: MAIN RESULTS}
			\label{sec:main_res}
			\noindent In this section we formally state the main results concerning the distributed parameter estimation $\mathcal{CIRFE}$ algorithm. The proofs are relegated to Section \ref{sec:proof_mainres}. The first result concerns with the consistency of the parameter estimate sequence at each agent $n$.
			\begin{Theorem}
				\label{th1}
				Consider the parameter estimate sequence $\{\widetilde{\mathbf{x}}(t)\}$ generated by the $\mathcal{CIRFE}$ algorithm according to \eqref{eq:cirfe_3}. Let Assumptions \ref{as:1}-\ref{as:6} hold. Then, we have,
				{\small\begin{align}
					\label{eq:th1_1}
					\mathbb{P}_{\btheta^{\ast}}\left(\lim_{t\to\infty}\widetilde{\mathbf{x}}(t)=\mathcal{P}\left(\mathbf{1}_{N}\otimes\btheta^{\ast}\right)\right)=1.
					\end{align}}
			\end{Theorem}
			\noindent  At this point, we note that the estimate sequence generated by $\mathcal{CIRFE}$ at any agent $n$ is strongly consistent, i.e., $\mathbf{x}_{n}(t)\rightarrow\btheta^{\ast}_{\mathcal{I}_{n}}$ almost surely~(a.s.) as $t\rightarrow\infty$. It is also to be noted that, owing to the heterogeneous objectives of the agents, the consensus in terms of the estimates sequences across any pair of agents is only limited to the common components of the parameter in their interest sets.\\
			\begin{Theorem}
					\label{th1.1}
					Let the hypothesis of theorem \ref{th1} hold. Then, we have,
					\begin{align}
					\label{eq:th1.1}
					\mathbb{E}\left[\left\|\widetilde{\mathbf{x}}(t)-\mathcal{P}\left(\mathbf{1}_{N}\otimes\boldsymbol{\theta}^{\ast}\right)\right\|^{2}\right] = O\left(\frac{1}{t}\right)
					\end{align}
				\end{Theorem}
				\noindent Thus, we note that the mean square error of the estimate sequence with respect to the components of the parameter $\boldsymbol{\theta}^{\ast}$ decays as $1/t$.
			\noindent With the above development in place, we state a result which allows us to benchmark the asymptotic efficiency of the proposed algorithm.
			\begin{Theorem}
				\label{th2}
				Let the hypotheses of Theorem \ref{th1} hold. Then, the time-scaled sequence $\sqrt{t+1}\left(\widetilde{\mathbf{x}}(t)-\mathcal{P}\left(\mathbf{1}_{N}\otimes\btheta^{\ast}\right)\right)$ is asymptotically normal, i.e.,
				{\small\begin{align}
					\label{eq:th2_1}
					\sqrt{t+1}\left(\widetilde{\mathbf{x}}(t)-\mathcal{P}\left(\mathbf{1}_{N}\otimes\btheta^{\ast}\right)\right)\overset{\mathcal{D}}{\Longrightarrow}\mathcal{N}\left(\mathbf{0}, \mathbf{S}_{R}\right),
					\end{align}}
				where
				{\small\begin{align}
					\label{eq:th2_2}
					&\mathbf{S}_{R}=\mathbf{P}\mathbf{M}\mathbf{P}^{\top}\nonumber\\
					&\left[\mathbf{M}\right]_{ij}=\left[\mathbf{P}\mathbf{Q}\left(\sum_{n=1}^{N}\mathcal{P}_{\mathcal{I}_n}\mathbf{H}_{n}^{\top}\mathbf{R}_{n}^{-1}\mathbf{H}_{n}\mathcal{P}_{\mathcal{I}_n}\right)\mathbf{Q}\mathbf{P}\right]_{ij}\nonumber\\&\times\left(\left[\mathbf{\Lambda}\right]_{ii}+\left[\mathbf{\Lambda}\right]_{jj}-1\right)^{-1},
					\end{align}}
				and $\mathbf{P}$ and $\mathbf{\Lambda}$ are orthonormal and diagonal matrices such that $\mathbf{P}^{\top}\mathbf{Q}\left(\sum_{n=1}^{N}\mathcal{P}_{\mathcal{I}_n}\mathbf{H}_{n}^{\top}\mathbf{R}_{n}^{-1}\mathbf{H}_{n}\mathcal{P}_{\mathcal{I}_n}\right)\mathbf{P} = \mathbf{\Lambda}$, in which, $\mathbf{Q} = \textit{diag}\left[\frac{1}{Q_{1}},\frac{1}{Q_{2}},\cdots,\frac{1}{Q_{N}}\right]$, with $Q_{i}$ denoting the number of agents interested in the $i$-th entry of $\btheta^{\ast}$.
			\end{Theorem}

			\noindent Theorem \ref{th2} establishes the asymptotic normality of the time-scaled (auxilliary)~estimate sequence. Noting that the estimate sequence $\{\mathbf{x}_{n}(t)\}$ is a linear transformation of the auxiliary estimate sequence, we conclude that $\sqrt{t+1}\left(\mathbf{x}_{n}(t)-\btheta^{\ast}_{\mathcal{I}_{n}}\right)$ is also asymptotically normal. It is also to be noted that, when the interest sets of each agent is the identity matrix, i.e., every agent is interested to reconstruct the entire parameter, the matrix $\mathbf{Q}$ reduces to $\frac{\mathbf{I}}{N}$ and the asymptotic covariance reduces to that of the classical consensus+innovations linear parameter estimation case~(see \cite{kar2011convergence} and the corresponding update in \eqref{eq:benchmark_ci}). In this sense, the classical linear parameter estimation case is a special case of the problem being addressed here. It is to be noted that the case in which $\mathbf{Q}$ reduces to $\frac{1}{\widetilde{Q}}\mathbf{I}$ for some $\widetilde{Q}<N$~($\widetilde{Q}<N$ agents interested in each entry of $\btheta^{\ast}$), the asymptotic covariance reduces to,
			{\small\begin{align*}
				\mathbf{S}_{R}=\frac{a\mathbf{I}}{2\widetilde{Q}}+\frac{\left(\frac{1}{N}\sum_{n=1}^{N}\mathbf{H}_{n}^{\top}\mathbf{R}_{n}^{-1}\mathbf{H}_{n}+\frac{\mathbf{I}}{2a}\right)^{-1}}{\widetilde{Q}}.
				\end{align*}}
			The asymptotic covariance as derived in Theorem \ref{th2} explicitly showcases the heterogeneity in the scaling with respect to different components of the parameter through $\mathbf{Q}$, as different components have different cardinalities of interest sets.\\
			The convergence rate is unaffected by the communication of low-dimensional estimates, i.e., the mean square error of the proposed scheme decays as $1/t$ as characterized by Theorem \ref{th1.1}. However, by communicating low dimensional estimates which is due to the interest sets being strict subsets of $\{1,2,\cdots,N\}$, the variance of the estimation scheme is affected in terms of scaling by the number of agents. In particular, as demonstrated by Theorem \ref{th2}, the variance of the estimate sequence scales inversely with the number of agents interested to reconstruct the particular entry. Thus, larger the size of the communicated estimates lower is the variance. For instance, the variance scaling as $1/N$ is obtained if every agent is interested to reconstruct the entire parameter. Intuitively speaking, the difference in scaling can be attributed to averaging by a smaller number of agents against averaging by the entire network. However, note that the scaling is only with respect to the asymptotic covariance and as we will demonstrate later in section \ref{sec:sim} on line graphs, the finite time variance of the error estimates can be lower for the proposed algorithm with respect to agents which directly do not observe the component of the parameter being estimated.
			
			\section{Simulation Results}
			\label{sec:sim}
			\noindent In this section, we demonstrate the efficiency of the proposed algorithm $\mathcal{CIRFE}$ through simulation experiments on a synthetic dataset. In particular, we construct a $10$ node ring network, where every agent has exactly two nodes in its communication neighborhood. We number the nodes from $1$ to $10$. The neighbors for the $i$-th node in the communication graph are the nodes $(i-1)mod~10$ and $(i+1)mod~10$.
			
			\noindent The physical coupling which affects each agent's observations is assumed to be an agent's $2$-hop neighborhood. For instance, node $1$'s observations are affected by the value of the field at nodes $9$, $10$, $2$ and $3$. Thus, $\widetilde{\mathcal{I}}_{1} = \{9,10,2,3\}$. The interest set of each agent is taken to be all the field values which affects its observation. For instance, $\mathcal{I}_{1}=\{9,10,1,2,3\}$. We resort to a static Laplacian in the simulation setup here. We also note that in this case the inter-agent communication network is sparser than the physical network induced by measurement coupling. Each agent makes a scalar observation at each time. Hence, the observation matrix for each agent is given by a $5$-sparse $10$-dimensional row vector. To be specific, the observation matrices used in the simulation setup are given by $\mathbf{H}_{1}= [1.0,1.2,1.3,0,0,0,0,0,1.4,1.5]$, $\mathbf{H}_2=[1.5,1.0,1.2,1.3,0,0,0,0,0,1.4]$, $\mathbf{H}_3=[1.4,1.5,1.0,1.2,1.3,0,0,0,0,0]$, $\mathbf{H}_4=[0,1.4,1.5,1.0,1.2,1.3,0,0,0,0]$, $\mathbf{H}_5 = [0,0,1.4,1.5,1.0,1.2,1.3,0,0,0]$, $\mathbf{H}_6 = [0,0,0,1.4,1.5,1.0,1.2,1.3,0,0]$, $\mathbf{H}_7 = [0,0,0,0,1.4,1.5,1.0,1.2,1.3,0]$, $\mathbf{H}_8 = [0,0,0,0,0,1.4,1.5,1.0,1.2,1.3]$, $\mathbf{H}_9 = [1.3,0,0,0,0,0,1.4,1.5,1.0,1.2]$ and $\mathbf{H}_{10} = [1.2,1.3,0,0,0,0,0,1.4,1.5,1.0]$. The noise covariance $\mathbf{R}$ is taken to be $\mathbf{I}_{10}$. The parameter capturing the field values is taken to be $\btheta=[1.2,1.3,1.4,0.8,0.7,1.1,0.9,1.0,1.8,0.6]$. It can be seen that Assumption \ref{as:5} is satisfied, by verifying Lemma \ref{str_RLU} for the third parameter component $\theta^{\ast}_{3}$. \\
			We carry out $500$ Monte-Carlo simulations for analyzing the convergence of the parameter estimates. The estimates are initialized as $\mathbf{x}_{n}(0)=\mathbf{0}$ for $n=1,\cdots, 10$. The normalized error for the $n$-th agent at time $t$ is given by the quantity $\left\|\mathbf{x}_{n}(t)-\mathcal{P}_{\mathcal{I}_n}\btheta\right\|/5$, as each agent's interest set has the cardinality of $5$. Figure \ref{fig:2} shows the normalized error at every agent against the time index $t$.
			\begin{figure}
				\centering
				\captionsetup{justification=centering}
				\includegraphics[width=60mm]{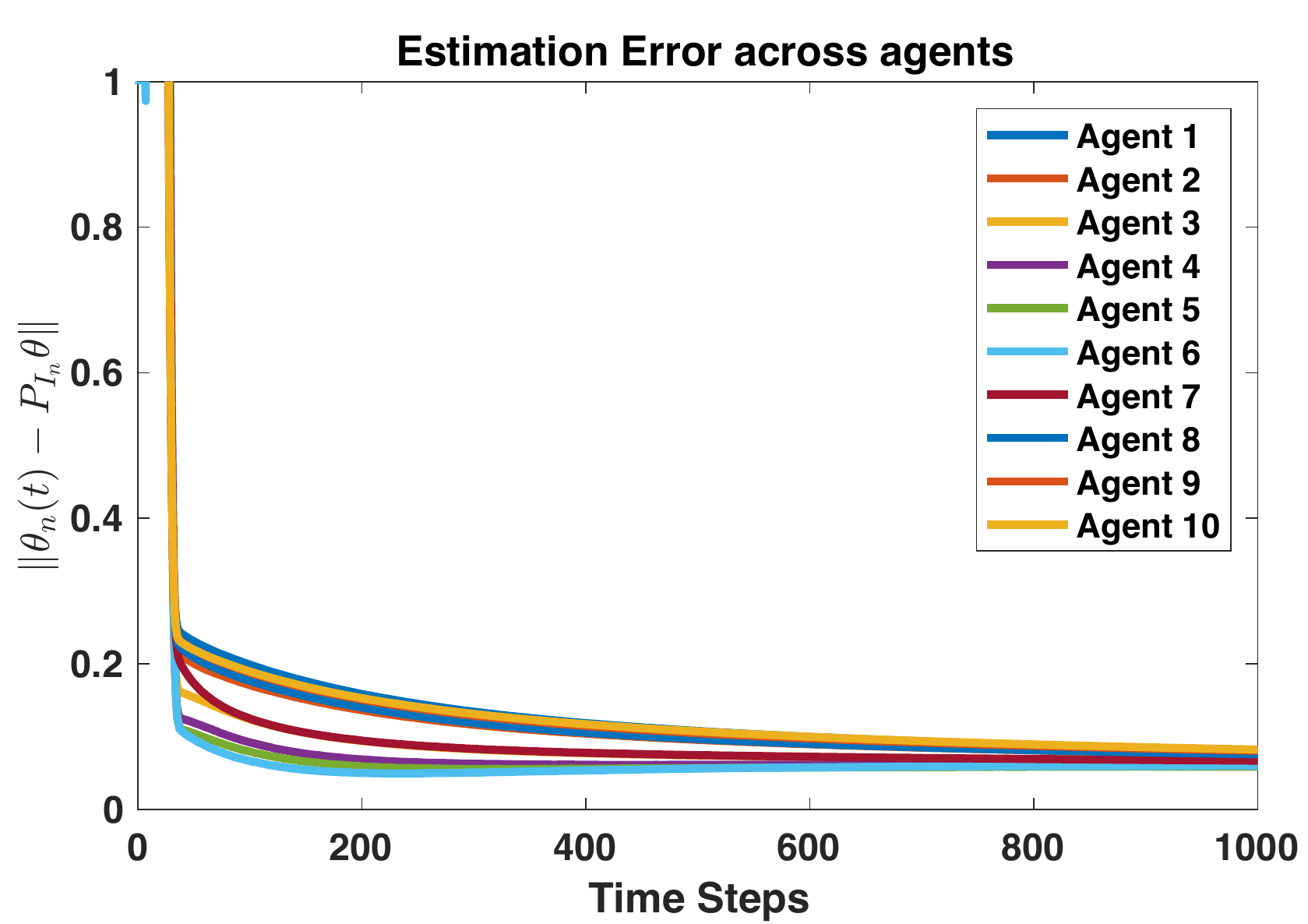}
				\caption{Convergence of normalized estimation error at each agent}\label{fig:2}
			\end{figure}
			\noindent In Figures \ref{fig:3} and \ref{fig:4} we compare the performance of $\mathcal{CIRFE}$ to the classical distributed estimator in \cite{kar2011convergence}~(see \eqref{eq:benchmark_ci} for the corresponding update), where each agent is interested in reconstructing the entire state or the parameter vector. We refer to the estimates of the distributed estimator in \cite{kar2011convergence} as ``classical" and ``classical-d"~(to be specified shortly) in the sequel.
			\begin{figure}
				\centering
				\captionsetup{justification=centering}
				\includegraphics[width=60mm]{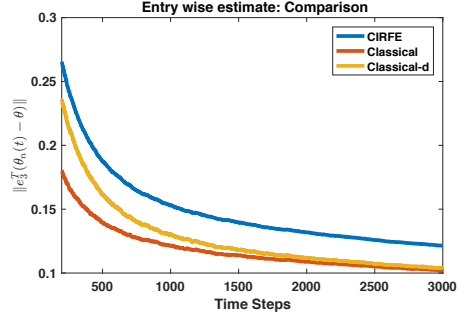}
				\caption{Comparison of $e_{3}^{\top}\btheta^{\ast}$ estimation error}\label{fig:3}
			\end{figure}
			\begin{figure}
				\centering
				\captionsetup{justification=centering}
				\includegraphics[width=60mm]{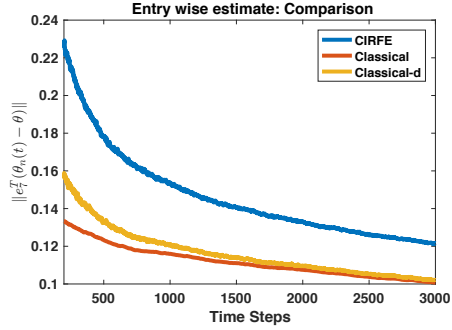}
				\caption{Comparison of $e_{7}^{\top}\btheta^{\ast}$ estimation error}\label{fig:4}
			\end{figure}
			\noindent In Figures \ref{fig:3} and \ref{fig:4}, ``Classical-d" represents the case in the algorithm in \cite{kar2011convergence}, where an agent does not observe the entry to be estimated and entirely depends on the neighborhood communication to estimate the quantity of interest. We specifically study the estimation performance of the agents in the ``Classical-d" case, as these are the agents that tend to increase the communication overhead considerably by being interested in estimates of components that they do not directly observe, relying on other agents possibly far off to obtain the desired information. Note that, in the current simulation setup, such class of agents do not exist for the proposed $\mathcal{CIRFE}$ algorithm. It can be observed from figures \ref{fig:3} and \ref{fig:4} that the estimation error in $\mathcal{CIRFE}$ is higher than that of the classical distributed estimator but at the same time exchanging $5$-dimensional or even smaller dimensional messages as opposed to $10$-dimensional messages in the case of the classical consensus+innovations estimator in \cite{kar2011convergence}. This analysis brings about an inherent trade-off between estimation error and the dimension of the messages exchanged between agents. It is also to be noted that the agents in case of $\mathcal{CIRFE}$ store $5$-dimensional vectors at each time step as opposed to $10$-dimensional vectors in the case of the classical. An intuitive way to interpret the higher estimation error is noting the fact that, effective for the algorithm $\mathcal{CIRFE}$, the estimation procedure for each entry of the parameter $\btheta^{*}$ effectively happens over a line graph, whereas for the Classical and ``Classical-d" procedures the communication graph to which the estimation procedure conforms to is a ring graph. In order to demonstrate the effectiveness of the algorithm $\mathcal{CIRFE}$, we consider a line graph, where the agents have the same sensing model as in the previous case except for the two edges of the line graph. Thus, agent $1$ and $10$'s observations are dependent on agent $2$ and agent $9$'s state. Furthermore, we assume that each agent's observation is physically coupled with the states of the agents' in its one-hop neighborhood. The interest set for the $1$st and $10$th agents are taken to be $\{1,2\}$ and $\{9,10\}$ respectively. All the other agents, have interest sets of cardinality three, i.e, itself and its one-hop neighborhood.
				\noindent In Figures \ref{fig:5} and \ref{fig:6} we compare the performance of $\mathcal{CIRFE}$ to the classical distributed estimator in \cite{kar2011convergence}~(see \eqref{eq:benchmark_ci} for the corresponding update), with the aforementioned line graph setup.
				\begin{figure}
					\centering
					\captionsetup{justification=centering}
					\includegraphics[width=60mm]{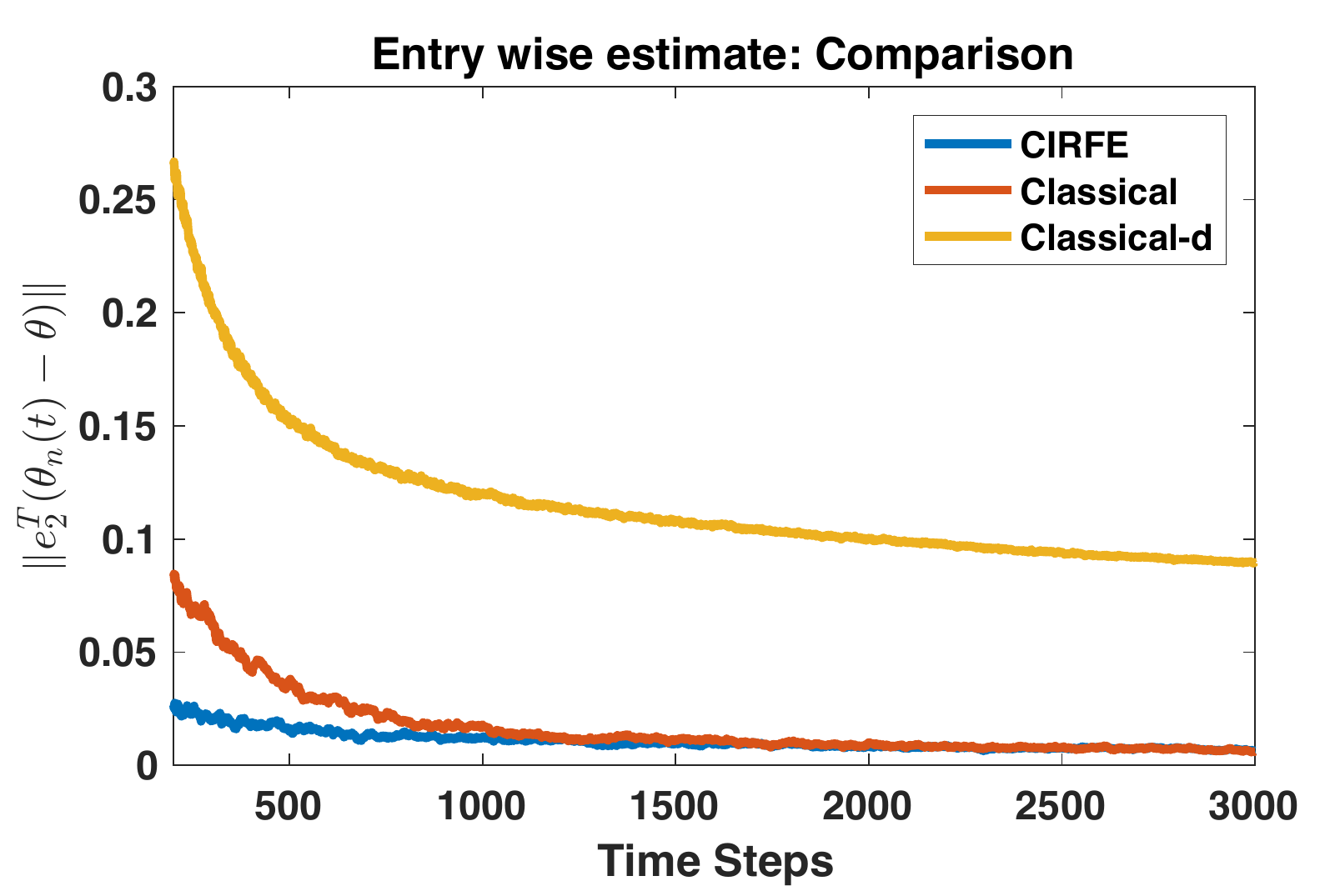}
					\caption{Comparison of $e_{2}^{\top}\btheta^{\ast}$ estimation error}\label{fig:5}
				\end{figure}
				\begin{figure}
					\centering
					\captionsetup{justification=centering}
					\includegraphics[width=60mm]{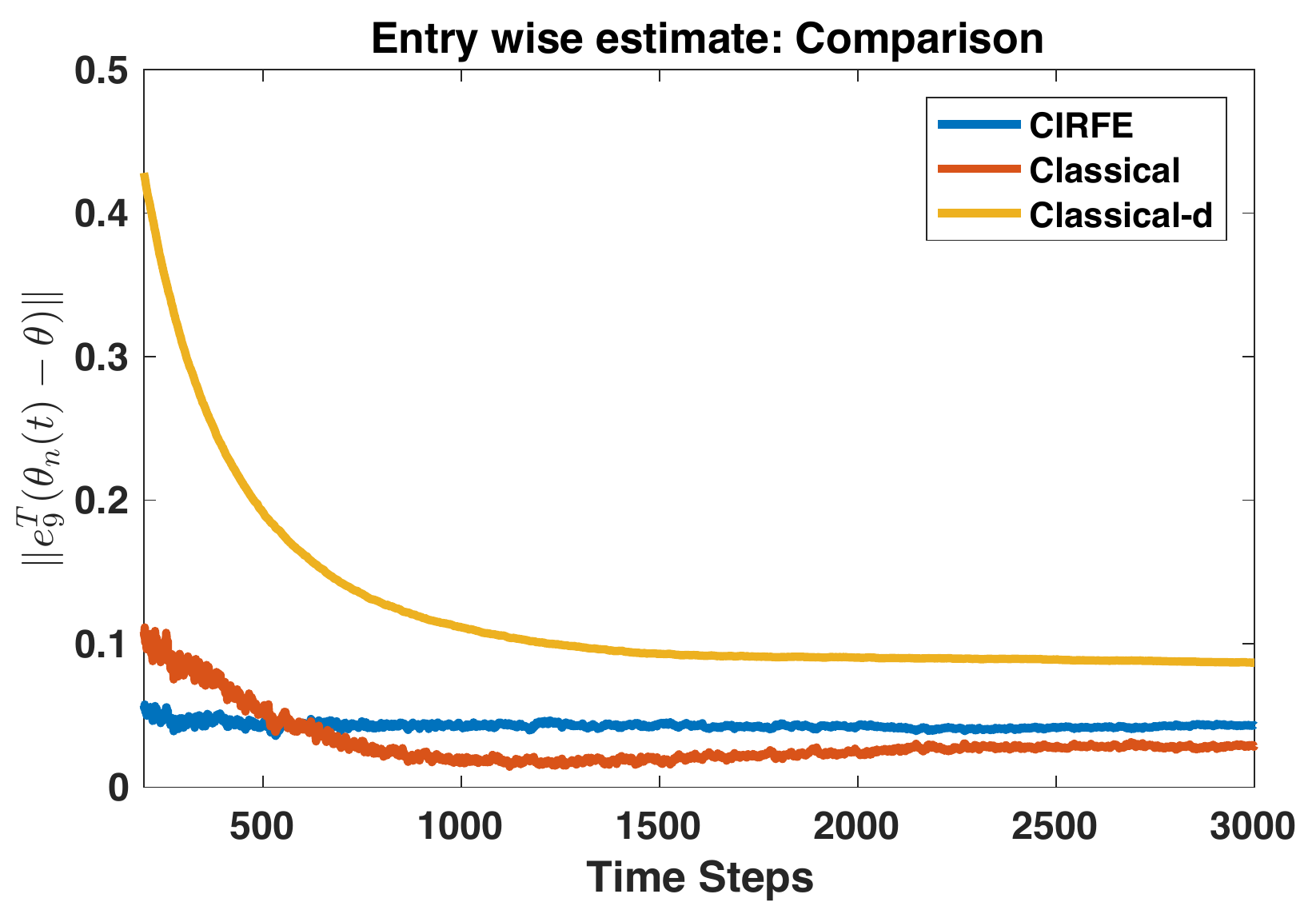}
					\caption{Comparison of $e_{9}^{\top}\btheta^{\ast}$ estimation error}\label{fig:6}
				\end{figure}
				For the ``classical-d" case, the agent selected was the farthest end of the graph. It is well known that under a line graph, the performance of a distributed protocol is affected due to poor connectivity. It can be seen from figures \ref{fig:5} and \ref{fig:6} that the performance of $\mathcal{CIRFE}$ closely resembles that of the classical benchmark algorithm with respect to an agent which observes the particular entry. However, for agents far away from the agent which observes the particular entry, $\mathcal{CIRFE}$ outperforms them. Intuitively speaking, while in this case, the communication protocol for each entry of the parameter in $\mathcal{CIRFE}$ conforms to a line graph, where the maximum number of vertices in each is $3$, for the benchmark the line graph consists of $10$ agents. In order to reinforce the effectiveness of $\mathcal{CIRFE}$, we ran experiments on a $30$ node line graph, where each agent except the nodes numbered $1$, $2$, $29$ and $30$, have an interest set of cardinality $5$. The nodes numbered $1$, $2$, $29$ and $30$ are assumed to have interest sets of cardinality $3$, $4$, $4$ and $3$ respectively. For instance the interest sets of agents $1$ and $2$ are given by $\{1,2,3\}$ and $\{1,2,3,4\}$ respectively. We assume that the physical coupling which affects each agent's observation is limited to its two-hop neighborhood.
				\noindent In Figures \ref{fig:7} and \ref{fig:8} we compare the performance of $\mathcal{CIRFE}$ to the classical distributed estimator in \cite{kar2011convergence}~(see \eqref{eq:benchmark_ci} for the corresponding update), with the aforementioned line graph setup.
				\begin{figure}
					\centering
					\captionsetup{justification=centering}
					\includegraphics[width=60mm]{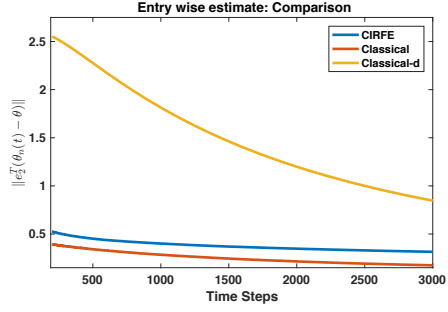}
					\caption{Comparison of $e_{2}^{\top}\btheta^{\ast}$ estimation error}\label{fig:7}
				\end{figure}
				\begin{figure}
					\centering
					\captionsetup{justification=centering}
					\includegraphics[width=60mm]{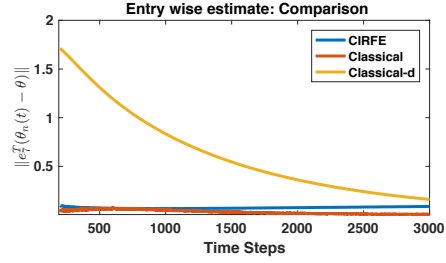}
					\caption{Comparison of $e_{7}^{\top}\btheta^{\ast}$ estimation error}\label{fig:8}
				\end{figure}
				For the ``classical-d" case, the agent selected was the farthest end of the graph as in the previous case. It can be seen from figures \ref{fig:7} and \ref{fig:7} that the performance of $\mathcal{CIRFE}$ closely resembles that of the classical benchmark algorithm with respect to an agent which observes the particular entry. However, for agents far away from the agent which observes the particular entry, $\mathcal{CIRFE}$ outperforms them.\\
				Technically speaking, in the classical case, an agent which is diameter number of steps away from a particular agent requires diameter number of time steps to fuse information from the other agent for an entry which it does not observe. In contrast with the classical case, the estimation of a particular entry of the parameter effectively happens over the induced subgraph with respect to the particular entry which typically will have smaller diameter as compared to the original graph. In conclusion, forcing an agent to obtain estimates of all parameter components may actually slow down the overall process in many scenarios of interest (especially situations involving large graphs with poor connectivity), as some of these components are only observed at agents geographically distant from the agent under consideration.
			
			\section{Proof of Main Results}
			\label{sec:proof_mainres}
			\begin{proof}[Proof of Theorem \ref{th1}]
				Define the sequence, $\{\widehat{\mathbf{x}}(t)\}$, as $\widehat{\mathbf{x}}(t) = \widetilde{\mathbf{x}}(t)-\mathcal{P}\left(\mathbf{1}_{N}\otimes\btheta^{\ast}\right)$.
				Then, we have,
				{\small\begin{align}
					\label{eq:tr1pr_1}
					&\widehat{\mathbf{x}}(t+1)=\widehat{\mathbf{x}}(t) - \left(\beta_{t}\overline{\mathbf{L}}_{\mathcal{P}}+\alpha_{t}\mathcal{P}\mathbf{G}_{H}\mathbf{R}^{-1}\mathbf{G}_{H}^{\top}\mathcal{P}\right)\widehat{\mathbf{x}}(t)\nonumber\\
					&-\beta_{t}\widetilde{\mathbf{L}}_{\mathcal{P}}(t)\widehat{\mathbf{x}}(t)+\alpha_{t}\mathcal{P}\mathbf{G}_{H}\mathbf{R}^{-1}\left(\mathbf{y}(t)-\mathbf{G}_{H}^{\top}\mathcal{P}\left(\mathbf{1}_{N}\otimes\btheta^{*}\right)\right).
					\end{align}}
				It is clear that $\{\widehat{\mathbf{x}}(t)\}$ is Markov with respect to its natural filtration $\{\mathcal{F}^{\widehat{\mathbf{X}}}_{t}\}$. Now, define the function $V:\mathbb{R}^{N^{2}}\longmapsto\mathbb{R}_{+}$ as, $V(\mathbf{y})=\left\|\mathbf{y}\right\|^{2}$, for all $\mathbf{y}$.
				We note that
				{\small\begin{align}
					\label{eq:tr1pr_2}
					\mathbb{E}_{\mathbf{\theta}^{\ast}}\left[V(\widehat{\mathbf{x}}(t+1))~|~\mathcal{F}^{\widetilde{\mathbf{x}}}_{t}\right]=\mathbb{E}_{\mathbf{\theta}^{\ast}}\left[V(\widehat{\mathbf{x}}(t+1))~|~\widehat{\mathbf{x}}(t)\right]
					\end{align}}
				By basic algebraic manipulations, we have,
				{\small\begin{align}
					\label{eq:tr1pr_3}
					&\mathbb{E}_{\mathbf{\theta}^{\ast}}\left[V(\widehat{\mathbf{x}}(t+1))~|~\widehat{\mathbf{x}}(t)\right] \nonumber\\&\leq  \widehat{\mathbf{x}}(t)^{\top}\left(\mathbf{I}-\beta_{t}\overline{\mathbf{L}}_{\mathcal{P}}-\alpha_{t}\mathcal{P}\mathbf{G}_{H}\mathbf{R}^{-1}\mathbf{G}_{H}^{\top}\mathcal{P}\right)^{2}\widehat{\mathbf{x}}(t)\nonumber\\&+\beta_{t}^{2}\mathbb{E}_{\mathbf{\theta}^{\ast}}\left[\left\|\widetilde{\mathbf{L}}_{\mathcal{P}}(t)\widehat{\mathbf{x}}(t)\right\|^{2}\right]\nonumber \\ & +\alpha^{2}_t\mathbb{E}_{\mathbf{\theta}^{\ast}}\left[\left\|\mathcal{P}\mathbf{G}_{H}\mathbf{R}^{-1}\left(\mathbf{y}(t)-\mathbf{G}_{H}^{\top}\mathcal{P}\left(\mathbf{1}_{N}\otimes\btheta^{*}\right)\right)\right\|^{2}\right].
					\end{align}}
				We note that $\beta_{t}\overline{\mathbf{L}}_{\mathcal{P}}+\alpha_{t}\mathcal{P}\mathbf{G}_{H}\mathbf{R}^{-1}\mathbf{G}_{H}^{\top}\mathcal{P}$ is uniformly elliptic on the subspace $\mathcal{S}_{\mathcal{P}}$, and it is precisely the subspace where $\{\widehat{\mathbf{x}}(t)\}$ resides. We thus prove the result by showing convergence to zero of the sequence $\{\widehat{\mathbf{x}}(t)\}$ through the subspace $\mathcal{S}_{\mathcal{P}}$. To this end, using the fact, that, for $\mathbf{y}\in\mathcal{S}_{\mathcal{P}}$,
				{\small\begin{align}
					\label{eq:tr1pr_4}
					\mathbf{y}^{\top}\left(\frac{\beta_0}{\alpha_0}\mathbf{L}_{\mathcal{P}}+\mathcal{P}\mathbf{G}_{H}\mathbf{R}^{-1}\mathbf{G}_{H}^{\top}\mathcal{P}\right)\mathbf{y}\geq c_{1}\left\|\mathbf{y}\right\|^{2},~\mbox{a.s.}
					\end{align}}
				By choosing, $t_{1}$ sufficiently large, we have for $\widehat{\mathbf{x}}(t)^{\top}\in\mathcal{S}_{\mathcal{P}}$ for all $t\geq t_1$,
				{\small\begin{align}
					\label{eq:tr1pr_5}
					&\widehat{\mathbf{x}}(t)^{\top}\left(\beta_{t}^{2}\overline{\mathbf{L}}_{\mathcal{P}}^{2}+\beta_{t}^{2}\mathbb{E}_{\mathbf{\theta}^{\ast}}\left\|\widetilde{\mathbf{L}}_{\mathcal{P}}(t)\right\|^{2}-\beta_{t}\overline{\mathbf{L}}_{\mathcal{P}}\right)\widehat{\mathbf{x}}(t)\nonumber\\
					&\le \left(c^{'}_1\beta_{t}^2-c^{'}_3\beta_{t}\right)\left\|\widehat{\mathbf{x}}(t)\right\|^{2} \le 0,
					\end{align}}
				where equality exists if $\widehat{\mathbf{x}}(t)=\mathcal{P}\left(\mathbf{1}_{N}\otimes \mathbf{a}\right)$, where $\mathbf{a}\in\mathbb{R}^{N}$.
				Thus, we obtain the following inequality:
				{\small\begin{align}
					\label{eq:thr1:6}
					&\mathbb{E}_{\mathbf{\theta}^{\ast}}\left[V(\widehat{\mathbf{x}}(t+1))~|~\widehat{\mathbf{x}}(t)=\mathbf{y}\right]-V(\mathbf{y})\leq c_{11}\alpha_{t}^{2}\left(1+\left\|\mathbf{y}\right\|^{2}\right)\nonumber\\&-\alpha_{t}c_{10}\left\|\mathbf{y}\right\|^{2}
					\end{align}}
				for all $\mathbf{y}\in\mathcal{S}_{\mathcal{P}}$.
				Now, define the function $W:\mathbb{T}_{+}\times\mathbb{R}^{N^{2}}\longmapsto\mathbb{R}_{+}$:
				{\small\begin{align}
					\label{eq:thr1:7}
					W(t,\mathbf{y})=\left(1+V(\mathbf{y})\right)\prod_{j=t}^{\infty}(1+c_{11}\alpha^{2}_{j}).
					\end{align}}
				From \eqref{eq:thr1:6} it can be shown that, for $\mathbf{y}\in\mathcal{S}_{\mathcal{P}}$,
				{\small\begin{align}
					\label{eq:thr1:8}
					&\mathbb{E}_{\mathbf{\theta}^{\ast}}\left[W(t+1,\widehat{\mathbf{x}}(t+1))~|~\widehat{\mathbf{x}}(t)=\mathbf{y}\right]-W(t,\mathbf{y})\nonumber\\&\leq -\alpha_{t}c_{10}\left\|\mathbf{y}\right\|^{2}\left(\prod_{j=t+1}^{\infty}(1+c_{11}\alpha^{2}_j)\right)\nonumber \\ & \leq -\alpha_{t}c_{10}\left\|\mathbf{y}\right\|^{2}
					\end{align}}
				Now consider $\varepsilon>0$, and let $V_{\varepsilon}$ denote the set
				{\small\begin{align}
					\label{eq:thr1:9}
					V_{\varepsilon}=\{\mathbf{y}\in\mathbb{R}^{N^{2}}~|~\left\|\mathbf{y}\right\|\geq\varepsilon\}\cap\mathcal{S}_{\mathcal{P}}
					\end{align}}
				Also, define $\tau_{\varepsilon}$ to be the exit time of the process $\{\widehat{\mathbf{x}}(t)\}$ from $V_{\varepsilon}$, i.e.,
				{\small\begin{align}
					\label{eq:thr1:10}
					\tau_{\varepsilon}=\inf\{i\in\mathbb{T}_{+}~|~\widehat{\mathbf{x}}(t)\notin V_{\varepsilon}\}
					\end{align}}
				We now show that $\tau_{\varepsilon}<\infty$ a.s. For mathematical simplicity, assume $\widehat{\mathbf{x}}(0)\in V_{\varepsilon}$. Consider the function
				{\small\begin{align}
					\label{thr1:11}
					\widetilde{W}(t,\mathbf{y})=W(t,\mathbf{y})+c_{10}\varepsilon^{2}\sum_{j=0}^{t-1}\alpha_j
					\end{align}}
				By \eqref{eq:thr1:8} it follows that, for $\mathbf{y}\in V_{\varepsilon}$,
				{\small\begin{align}
					\label{thr1:12}
					\mathbb{E}_{\mathbf{\theta}^{\ast}}\left[W(t+1,\widehat{\mathbf{x}}(t+1))~|~\widehat{\mathbf{x}}(t)=\mathbf{y}\right]-W(t,\mathbf{y})\leq -\alpha_{t}c_{10}\varepsilon^{2}
					\end{align}}
				and hence, it can be shown that, for $\mathbf{y}\in V_{\varepsilon}$,
				{\small\begin{align}
					\label{thr1:13}
					\mathbb{E}_{\mathbf{\theta}^{\ast}}\left[\widetilde{W}(t+1,\widehat{\mathbf{x}}(t+1))~|~\widehat{\mathbf{x}}(t)=\mathbf{y}\right]-\widetilde{W}(t,\mathbf{y})\leq 0
					\end{align}}
				Hence, we have that the stopped process $\{\widetilde{W}(\max\{t,\tau_{\varepsilon}\},\widehat{\mathbf{x}}(\max\{t,\tau_{\varepsilon}\}))\}$ is a super martingale. Being nonnegative it converges a.s. as $t\rightarrow\infty$. By \eqref{thr1:12}, we then conclude that the following term converges,
				{\small\begin{align}
					\label{thr1:14}
					\lim_{t\rightarrow\infty}c_{10}\varepsilon^{2}\sum_{j=0}^{(t\wedge\tau_{\varepsilon})-1}\alpha_j~\mbox{converges a.s.}
					\end{align}}
				Since, $\sum_{t\in\mathbb{T}_{+}}\alpha_t=\infty$, the above is possible, only if, $\tau_{\varepsilon}<\infty$ a.s.
				
				\noindent We thus note, that the process $\{\widehat{\mathbf{x}}(t)\}$ leaves the set $V_{\varepsilon}$ almost surely in finite time. Since, the process is constrained to lie in $\mathcal{S}_{\mathcal{P}}$ at all times, the finite time exit from $V_{\varepsilon}$ suggests,
				{\small\begin{align}
					\label{thr1:15}
					\mathbb{P}_{\mathbf{\theta}^{\ast}}\left(\inf\{t\in\mathbb{T}_{+}~|~\left\|\widehat{\mathbf{x}}(t)\right\|<\varepsilon\}<\infty\right)=1
					\end{align}}
				Since $\varepsilon>0$ is arbitrary, a subsequence almost surely converges to zero, and we have
				{\small\begin{align}
					\label{thr1:16}
					\mathbb{P}_{\mathbf{\theta}^{\ast}}\left(\liminf_{t\rightarrow\infty}\left\|\widehat{\mathbf{x}}(t)\right\|=0\right)=1
					\end{align}}
				Now going back to \eqref{eq:thr1:6} and noting that $\{\widehat{\mathbf{x}}(t)\}$ takes values in $\mathcal{S}_{\mathcal{P}}$, we conclude that the process $\{V(\widehat{\mathbf{x}}(t))\}$ is a nonnegative supermartingale. Hence,
				{\small\begin{align}
					\label{thr1:17}
					\mathbb{P}_{\mathbf{\theta}^{\ast}}\left(\lim_{t\rightarrow\infty}V(\widehat{\mathbf{x}}(t))~\mbox{exists}\right)=1
					\end{align}}
				Also, by \eqref{thr1:16}
				{\small\begin{align}
					\label{thr1:18}
					\mathbb{P}_{\mathbf{\theta}^{\ast}}\left(\liminf_{t\rightarrow\infty}V(\widehat{\mathbf{x}}(t))=0\right)=1
					\end{align}}
				and we conclude that
				{\small\begin{align}
					\label{thr1:19}
					\mathbb{P}_{\mathbf{\theta}^{\ast}}\left(\lim_{i\rightarrow\infty}\left\|\widehat{\mathbf{x}}(t)\right\|=0\right)=1
					\end{align}}
			\end{proof}
			
			\begin{proof}[Proof of Theorem \ref{th1.1}]
					From \eqref{eq:tr1pr_2}-\eqref{eq:tr1pr_4} in the proof of Theorem \ref{th1} we have, for $t\geq t_1$ ($t_1$ chosen appropriately large) and using the property that $\widehat{\mathbf{x}}(t)$ resides in $\mathcal{S}_{\mathcal{P}}$
					\begin{align*}
					&\mathbb{E}_{\mathbf{\theta}^{\ast}}\left[V(\widehat{\mathbf{x}}(t+1))~|~\widehat{\mathbf{x}}(t)\right]\leq  \left(1-c_{1}\alpha_{t}\right)\left\|\widehat{\mathbf{x}}(t)\right\|^{2}+\alpha^{2}_t c_{2} \nonumber\\
					&\Rightarrow\mathbb{E}_{\mathbf{\theta}^{\ast}}\left[\left\|\widehat{\mathbf{x}}(t+1)\right\|^{2}\right]\leq  \left(1-c_{1}\alpha_{t}\right)\left\|\widehat{\mathbf{x}}(t)\right\|^{2}+\alpha^{2}_t c_{2}\nonumber\\
					&\Rightarrow \mathbb{E}\left[\left\|\widetilde{\mathbf{x}}(t)-\mathcal{P}\left(\mathbf{1}_{N}\otimes\boldsymbol{\theta}^{\ast}\right)\right\|^{2}\right] = O\left(\frac{1}{t}\right).
					\end{align*}
					for appropriately chosen constants $c_1$ and $c_2$, where the conclusion in the last line follows from Lemma \ref{int_res_0}.
				\end{proof}
				\begin{proof}[Proof of Theorem \ref{th2}]
					\noindent Let the number of agents interested in the $i$-th entry of $\btheta^{\ast}$ be $Q_{i}$. To get the vector of estimates of the $i$-th entry of $\btheta^{\ast}$, left multiply the selector matrix $\mathcal{S}_{i}\in\mathbb{R}^{Q_{i}\times N^{2}}$ and noting that $S_{i}\mathbf{L}_{\mathcal{P}}(t)\widetilde{\mathbf{x}}(t) = \mathbf{L}_{\mathcal{P},i}(t)\widetilde{\mathbf{x}}(i,t)$, where $\mathbf{L}_{\mathcal{P},i}(t)\in \mathbb{R}^{Q_{i}\times Q_{i}} $ is the subgraph induced by the interest sets for the $i$-th entry of $\btheta^{\ast}$, which is connected as a result of a sufficient condition which enforced Assumption \ref{as:5} and $\widetilde{\mathbf{x}}(i,t)\in\mathbb{R}^{Q_{i}}$ is the vector of estimates for the $i$-th  entry of $\btheta^{\ast}$.
					
					\noindent A vector $\mathbf{z}\in\mathbb{R}^{N^{2}}$ may be decomposed as $\mathbf{z} = \mathbf{z}_{\mathcal{C}} + \mathbf{z}_{\mathcal{C}^{\perp}}$ with $\mathbf{z}_{\mathcal{C}}$ denoting its projection on the consensus or agreement subspace $\mathcal{C}$, $\mathcal{C}= \left\{\mathbf{z}\in\mathbb{R}^{N^{2}}|\mathbf{z}=\mathbf{1}_{N}\otimes\mathbf{a}~\textit{for~some}~\mathbf{a}\in\mathbb{R}^N\right\}$.
					\noindent We first prove the following Lemma regarding the mean connectedness of the subgraphs $\mathbf{L}_{\mathcal{P},i}(t)$.
					\begin{Lemma}
						\label{main_res_lap}
						Let $\left\{\mathbf{z}_{t}\right\}$ be an $\mathbb{R}^{N^{2}}$ valued $\mathcal{F}_{t}$-adapted process such that $\mathbf{z}_{t}\in\mathcal{C}^{\perp}$ for all $t$. Also, let $\left\{\mathbf{L}_{t}\right\}$ be an i.i.d. sequence of Laplacian matrices as in assumption \ref{as:4} that satisfies
						{\small\begin{align}
							\label{eq:mrl1}
							\lambda_{2}\left(\overline{\mathbf{L}}\right)=\lambda_{2}\left(\mathbb{E}\left[\mathbf{L}_{t}\right]\right) > 0,
							\end{align}}
						where $\mathbf{L}_t$ is $\mathcal{F}_{t+1}$-adapted and independent of $\mathcal{F}_{t}$ for all $t$.
						{\small\begin{align}
							\label{eq:l2_11}
							\left\|\left(\mathbf{I}_{N^{2}}-\left(\mathbf{L}(t)\otimes\mathbf{I}_{N}\right)\right)\mathbf{z}_{t}\right\|\le (1-r_{t})\left\|\mathbf{z}_{t}\right\|,
							\end{align}}
						where $\{r_t\}$ is a $\mathbb{R}^{+}$ valued $\mathcal{F}_{t+1}$ process satisfying
						{\small\begin{align}
							\label{eq:l2_10}
							\mathbb{E}\left[r_{t}|\mathcal{F}_{t}\right]\geq \underline{p}\beta_{t}\frac{\lambda_{2}\left(\overline{\mathbf{L}}\right)}{4|\mathcal{L}|},
							\end{align}}
						where $\mathcal{L}$ denotes the set of all possible Laplacians.
					\end{Lemma}
					\noindent The following Lemmas will be used to quantify the rate of convergence of distributed vector or matrix valued recursions to their network-averaged behavior.
					\noindent\begin{Lemma}
						\label{le:l1.1}
						Let $\{z_{t}\}$ be an $\mathbb{R}^{+}$ valued $\mathcal{F}_{t}$-adapted process that satisfies
						{\small\begin{align*}
							z_{t+1} \leq \left(1-r_{1}(t)\right)z_{t} +r_{2}(t)U_t(1+J_t),
							\end{align*}}
						\noindent where $\{r_1(t)\}$ is an $\mathcal{F}_{t+1}$-adapted process, such that for all $t$, $r_1(t)$ satisfies $0\leq r_1(t)\leq 1$ and
						{\small\begin{align*}
							a_{1}\leq \mathbb{E}\left[r_1(t)|\mathcal{F}_{t}\right] \leq \frac{1}{(t+1)^{\delta_{1}}}
							\end{align*}}
						\noindent with $a_{1} > 0$ and $0\leq \delta_{1} < 1$. The sequence $\{r_2(t)\}$ is deterministic and $\mathbb{R}^{+}$ valued and satisfies $r_2(t)\leq \frac{a_{2}}{(t+1)^{\delta_{2}}}$ with $a_2 > 0$ and $\delta_{2}>0$. Further, let $\{U_t\}$ and $\{J_t\}$ be $\mathbb{R}^{+}$ valued $\mathcal{F}_{t}$ and $\mathcal{F}_{t+1}$ adapted processes, respectively, with $\sup_{t\geq 0} \left\|U_t\right\|<\infty$ a.s. The process $\left\{J_{t}\right\}$ is i.i.d. with $J_{t}$ independent of $\mathcal{F}_{t}$ for each $t$ and satisfies the moment condition $\mathbb{E}\left[\left\|J_{t}\right\|^{2+\epsilon_{1}} \right]<\kappa<\infty$ for some $\epsilon_{1}>0$ and a constant $\kappa > 0$. Then, for every $\delta_{0}$ such that
						$0\leq \delta_0 < \delta_{2}-\delta_{1}-\frac{1}{2+\epsilon_{1}}$, we have $(t+1)^{\delta_{0}}z_{t}\to 0$ a.s. as $t\to\infty$.
					\end{Lemma}
					\begin{Lemma}[Lemma 4.1 in \cite{kar2013distributed}]
						\label{int_res_0}
						\noindent Consider the scalar time-varying linear system
						{\small\begin{align}
							\label{eq:int_res_0}
							u(t+1)\leq(1-r_{1}(t))u(t)+r_{2}(t),
							\end{align}}
						\noindent where $\{r_{1}(t)\}$ is a sequence, such that
						{\small\begin{align}
							\label{eq:int_res_0_1}
							\frac{a_{1}}{(t+1)^{\delta_{1}}}\leq r_{1}(t)\leq 1
							\end{align}}
						\noindent with $a_{1} >0, 0\leq\delta_{1}\leq 1$, whereas the sequence $\{r_{2}(t)\}$ is given by
						{\small\begin{align}
							\label{eq:int_res_0_2}
							r_{2}(t)\le\frac{a_{2}}{(t+1)^{\delta_{2}}}
							\end{align}}
						\noindent with $a_{2}>0, \delta_{2}\geq 0$. Then, if $u(0)\geq 0$ and $\delta_{1} < \delta_{2}$, we have
						{\small\begin{align}
							\label{eq:int_res_0_3}
							\lim_{t\to\infty}(t+1)^{\delta_0}u(t)=0,
							\end{align}}
						\noindent for all $0\le\delta_{0}<\delta_{2}-\delta_{1}$. Also, if $\delta_{1}=\delta_{2}$, then the sequence $\{u(t)\}$ stays bounded, i.e. $\sup_{t\geq 0}\left\|u(t)\right\|<\infty$.
					\end{Lemma}
					\begin{proof}[Proof of Lemma \ref{main_res_lap}]
						Let $\mathcal{L}$ denote the set of possible Laplacian matrices which is necessarily finite. Since the set of Laplacians is finite, we have,
						{\small\begin{align}
							\label{eq:l2_2}
							\underline{p}=\inf_{\mathbf{L}\in\mathcal{L}}p_{\mathbf{L}} > 0,
							\end{align}}
						with $p_{L}=\mathbb{P}\left(\mathbf{L}(t)=\mathbf{L}\right)$ for each $\mathbf{L}\in\mathcal{L}$ such that $\sum_{\mathbf{L}\in\mathcal{L}}p_{\mathbf{L}}=1$.
						We also have that $\lambda_{2}\left(\overline{\mathbf{L}}\right) > 0$ implies that for every $\mathbf{z}\in\mathcal{C}^{\perp}$, where,
						{\small\begin{align}
							\label{eq:l2_3}
							\mathcal{C} = \left\{\mathbf{x}|\mathbf{x}=\mathbf{1}_{N}\otimes \mathbf{a}, \mathbf{a}\in\mathbb{R}^{N}\right\},
							\end{align}}
						we have,
						{\small\begin{align}
							\label{eq:l2_4}
							\sum_{\mathbf{L}\in\mathcal{L}}\mathbf{z}^{\top}\mathbf{L}\mathbf{z}\geq \sum_{\mathbf{L}\in\mathcal{L}}\mathbf{z}^{\top}p_{\mathbf{L}}\mathbf{L}\mathbf{z}=\mathbf{z}^{\top}\overline{\mathbf{L}}\mathbf{z}\geq \lambda_{2}\left(\overline{\mathbf{L}}\right)\left\|\mathbf{z}\right\|^{2}.
							\end{align}}
						Owing to the finite cardinality of $\mathcal{L}$ and \eqref{eq:l2_4}, we also have that for each $\mathbf{z}\in\mathcal{C}^{\perp}$,$\exists \mathbf{L}_{\mathbf{z}}\in\mathcal{L}$ such that,
						{\small\begin{align}
							\label{eq:l2_41}
							\mathbf{z}^{\top}\mathbf{L}_{\mathbf{z}}\mathbf{z}\ge \frac{\lambda_{2}\left(\overline{\mathbf{L}}\right)}{|\mathcal{L}_{t}|}\left\|\mathbf{z}\right\|^{2}
							\end{align}}
						Moreover, since $\mathcal{L}$ is finite, the mapping $L_{\mathbf{z}}:\mathcal{C}^{\perp}\mapsto\mathcal{L}$ can be realized as a measurable function. For each, $\mathbf{L}\in\mathcal{L}$, the eigen values of $\mathbf{I}_{N^{2}}-\beta_{t}\left(\mathbf{L}\otimes\mathbf{I}_{N}\right)$ are given by $N$ repetitions of $1$ and $1-\beta_{t}\lambda_{n}\left(\mathbf{L}\right)$, where $2\leq n\leq N$.
						Thus, for $t\geq t_{0}$, $\left\|\mathbf{I}_{N^{2}}-\beta_{t}\left(\mathbf{L}\otimes\mathbf{I}_{N}\right)\right\| \leq 1$ and $\left\|\left(\mathbf{I}_{N^{2}}-\beta_{t}\left(\mathbf{L}\otimes\mathbf{I}_{N}\right)\right)\mathbf{z}\right\| \leq \left\|\mathbf{z}\right\|$. Hence, we can define a jointly measurable function $r_{\mathbf{L},\mathbf{z}}$ given by,
						{\small\begin{align}
							\label{l2_5}
							r_{\mathbf{L},\mathbf{z}} =
							\begin{cases}
							1 &~~\textit{if}~t<t_{0}~\textit{or}~\mathbf{z}=\mathbf{0}\\
							1-\frac{\left\|\left(\mathbf{I}_{NM}-\beta_{t}\left(\mathbf{L}\otimes\mathbf{I}_{M}\right)\right)\mathbf{z}\right\|}{\left\|\mathbf{z}\right\|} & ~~\textit{otherwise},
							\end{cases}
							\end{align}}
						which satisfies $0\le r_{\mathbf{L},\mathbf{z}} \le 1$ for each $\left(\mathbf{L},\mathbf{z}\right)$.
						\noindent Define $\{r_{t}\}$ to be a $\mathcal{F}_{t+1}$ process given by, $r_{t} = r_{\mathbf{L},\mathbf{z}_{t}}$ for each $t$ and $\left\|\left(\mathbf{I}_{N^2}-\beta_{t}\left(\mathbf{L}\otimes\mathbf{I}_{N}\right)\right)\mathbf{z}_{t}\right\|=(1-r_{t})\left\|\mathbf{z}_{t}\right\|$ a.s. for each $t$.
						\noindent Then, we have,
						{\small\begin{align}
							\label{eq:l2_6}
							&\left\|\left(\mathbf{I}_{N^2}-\beta_{t}\left(\mathbf{L}_{\mathbf{z}_{t}}\otimes\mathbf{I}_{N}\right)\right)\mathbf{z}_{t}\right\|^{2}\nonumber\\&=\mathbf{z}^{\top}_{t}\left(\mathbf{I}_{N^2}-2\beta_{t}\left(\mathbf{L}_{\mathbf{z}_{t}}\otimes\mathbf{I}_{N}\right)\right)\mathbf{z}_{t}\nonumber\\
							&+\mathbf{z}_{t}^{\top}\beta_{t}^2\left(\mathbf{L}_{\mathbf{z}_{t}}\otimes\mathbf{I}_{N}\right)^{2}\mathbf{z}_{t}\nonumber\\
							&\leq \left(1-2\beta_{t}\frac{\lambda_{2}\left(\overline{\mathbf{L}}\right)}{|\mathcal{L}|}\right)\left\|\mathbf{z}_{t}\right\|^{2}+c_{1}\beta_{t}^{2}\left\|\mathbf{z}_{t}\right\|^{2}\nonumber\\
							&\leq \left(1-\beta_{t}\frac{\lambda_{2}\left(\overline{\mathbf{L}}\right)}{|\mathcal{L}|}\right)\left\|\mathbf{z}_{t}\right\|^{2}
							\end{align}}
						where we have used the boundedness of the Laplacian matrix.
						\noindent With the above development in place, choosing an appropriate $t_{1}$~(making $t_{0}$ larger if necessary), for all $t\geq t_{1}$, we have,
						{\small\begin{align}
							\label{eq:l2_7}
							\left\|\left(\mathbf{I}_{N^2}-\beta_{t}\left(\mathbf{L}_{\mathbf{z}_{t}}\otimes\mathbf{I}_{N}\right)\right)\mathbf{z}_{t}\right\|\le \left(1-\beta_{t}\frac{\lambda_{2}\left(\overline{\mathbf{L}}\right)}{4|\mathcal{L}|}\right)\left\|\mathbf{z}_{t}\right\|^{2}.
							\end{align}}
						Then, from \eqref{eq:l2_7}, we have,
						{\small\begin{align}
							\label{eq:l2_8}
							&\mathbb{E}\left.\left[\left\|\left(\mathbf{I}_{N^2}-\beta_{t}\left(\mathbf{L}_{\mathbf{z}_{t}}\otimes\mathbf{I}_{N}\right)\right)\mathbf{z}_{t}\right\|\right|\mathcal{F}_{t}\right]\nonumber\\&=\sum_{\mathbf{L}\in\mathcal{L}}p_{\mathbf{L}}\left(1-r_{\mathbf{L},\mathbf{z}_{t}}\right)\left\|\mathbf{z}_{t}\right\|\nonumber\\
							&\le \left(1-\left(\underline{p}\beta_{t}\frac{\lambda_{2}\left(\overline{\mathbf{L}}\right)}{4|\mathcal{L}|}+\sum_{\mathbf{L}\neq\mathbf{L}_{\mathbf{z}_{t}}}\right)\right)\left\|\mathbf{z}_{t}\right\|.
							\end{align}}
						Since, $\sum_{\mathbf{L}\neq\mathbf{L}_{\mathbf{z}_{t}}} p_{\mathbf{L}}r_{\mathbf{L},\mathbf{z}_{t}}\geq 0$, we have for all $t\ge t_1$,
						{\small\begin{align}
							\label{eq:l2_9}
							&\left(1-\mathbb{E}\left[r_{t}|\mathcal{F}_{t}\right]\right)\left\|\mathbf{z}_{t}\right\| \nonumber\\&= \mathbb{E}\left.\left[\left\|\left(\mathbf{I}_{N^2}-\beta_{t}\left(\mathbf{L}_{\mathbf{z}_{t}}\otimes\mathbf{I}_{N}\right)\right)\mathbf{z}_{t}\right\|\right|\mathcal{F}_{t}\right]\nonumber\\
							&\le \left(1-\underline{p}\beta_{t}\frac{\lambda_{2}\left(\overline{\mathbf{L}}\right)}{4|\mathcal{L}|}\right)\left\|\mathbf{z}_{t}\right\|.
							\end{align}}
						\noindent As $r_{t}=1$ on the set $\{\mathbf{z}_{t}=0\}$, we have that,
						{\small\begin{align}
							\label{eq:l2_102}
							\mathbb{E}\left[r_{t}|\mathcal{F}_{t}\right]\geq \underline{p}\beta_{t}\frac{\lambda_{2}\left(\overline{\mathbf{L}}\right)}{4|\mathcal{L}|}.
							\end{align}}
						Thus, we have established that,
						{\small\begin{align}
							\label{eq:l2_112}
							\left\|\left(\mathbf{I}_{N^2}-\left(\mathbf{L}(t)\otimes\mathbf{I}_{N}\right)\right)\mathbf{z}_{t}\right\|\le (1-r_{t})\left\|\mathbf{z}_{t}\right\|,
							\end{align}}
						where $\{r_t\}$ is a $\mathbb{R}^{+}$ valued $\mathcal{F}_{t+1}$ process satisfying \eqref{eq:l2_102}.
					\end{proof}
					\noindent With the above development in place, consider the residual process $\{\mathbf{x}^{\dagger}(t)\}$ given by $\mathbf{x}^{\dagger}(i,t) = \widetilde{\mathbf{x}}(i,t)-\mathbf{1}_{Q_i}\otimes\widetilde{\mathbf{x}}_{\mbox{avg,i}}(t)$, where $i$ denotes the $i$-th entry of $\btheta^{\ast}$ and $\mathbf{x}^{\dagger}(t)=\left[\mathbf{x}^{\dagger}(1,t),\cdots,\mathbf{x}^{\dagger}(N,t)\right]^{\top}$. Thus, we have that the process $\{\mathbf{x}^{\dagger}(i,t)\}$ satisfies the recursion,
					{\small\begin{align}
						\label{eq:l2_12}
						\mathbf{x}^{\dagger}(i,t+1) = \left(\mathbf{I}_{Q_{i}}-\mathbf{L}_{\mathcal{P},i}(t)\right)\mathbf{x}^{\dagger}(i,t) +\alpha_{t}\widetilde{\mathbf{z}}(i,t),
						\end{align}}
					where the process $\{\widetilde{\mathbf{z}}(i,t)\}$ is given by
					{\small\begin{align}
						\label{eq:l2_13}
						\widetilde{\mathbf{z}}(i,t)=\left(\mathbf{I}_{Q_{i}}-\frac{1}{Q_i}\mathbf{1}_{Q_i}\mathbf{1}_{Q_i}^{\top}\right)\times\mathcal{S}_{i}\mathcal{P}\mathbf{G}_{H}\mathbf{R}^{-1}\left(\mathbf{y}(t)-\mathbf{G}_{H}^{\top}\mathcal{P}\widetilde{\mathbf{x}}(t)\right).
						\end{align}}
					\noindent From \eqref{eq:l2_13}, we also have,
					{\small\begin{align}
						\label{eq:l2_14}
						\widetilde{\mathbf{z}}(i,t)=\overline{\mathbf{J}}_{i,t}+\overline{\mathbf{U}}_{i,t},
						\end{align}}
					where,
					{\small\begin{align}
						\label{eq:l2_141}
						&\overline{\mathbf{J}}_{i,t}=\left(\mathbf{I}_{Q_{i}}-\frac{1}{Q_i}\mathbf{1}_{Q_i}\mathbf{1}_{Q_i}^{\top}\right)\nonumber\\&\times\mathcal{S}_{i}\mathcal{P}\mathbf{G}_{H}\mathbf{R}^{-1}\left(\mathbf{y}(t)-\mathbf{G}_{H}^{\top}\mathcal{P}\left(\mathbf{1}_{N}\otimes\btheta^{\ast}\right)\right)\nonumber\\
						&\overline{\mathbf{U}}_{t}=\left(\mathbf{I}_{Q_{i}}-\frac{1}{Q_i}\mathbf{1}_{Q_i}\mathbf{1}_{Q_i}^{\top}\right)\nonumber\\&\times\mathcal{S}_{i}\mathcal{P}\mathbf{G}_{H}\mathbf{R}^{-1}\left(\mathbf{G}_{H}^{\top}\mathcal{P}\left(\mathbf{1}_{N}\otimes\btheta^{\ast}\right)-\mathbf{G}_{H}^{\top}\mathcal{P}\widetilde{\mathbf{x}}(t)\right).
						\end{align}}
					By Theorem \ref{th1}, we also have that, the process $\{\widetilde{\mathbf{x}}(i,t)\}$ is bounded. Hence, there exists an $\mathcal{F}_{t}$-adapted process $\{\widetilde{U}_{i,t}\}$ such that $\left\|\overline{\mathbf{U}}_{i,t}\right\|\le \widetilde{U}_{i,t}$ and $\sup_{t\ge 0}\widetilde{U}_{i,t} < \infty$ a.s.. Furthermore, denote the process $U_{i,t}$ as follows,
					{\small\begin{align}
						\label{eq:l2_15}
						U_{i,t} = \max\left\{\widetilde{U}_{i,t}, \left\|\mathbf{I}_{Q_{i}}-\frac{1}{Q_i}\mathbf{1}_{Q_i}\mathbf{1}_{Q_i}^{\top}\right\|\right\}.
						\end{align}}
					With the above development in place, we conclude,
					{\small\begin{align}
						\label{eq:l2_15.1}
						\left\|\overline{\mathbf{U}}_{i,t}\right\|+\left\|\overline{\mathbf{J}}_{i,t}\right\| \le U_{i,t}\left(1+J_{i,t}\right),
						\end{align}}
					where $J_{i,t} = \left\|\mathbf{y}(t)-\mathbf{G}_{H}^{\top}\mathcal{P}\left(\mathbf{1}_{N}\otimes\btheta^{\ast}\right)\right\|$ and $\mathbb{E}_{\btheta}\left[J_{i,t}^{2+\epsilon}\right] < \infty$.
					Then, from \eqref{eq:l2_11}-\eqref{eq:l2_12} we have,
					{\small\begin{align}
						\label{eq:l2_16}
						\left\|\mathbf{x}^{\dagger}(i,t+1)\right\| \leq (1-r_{t})\left\|\mathbf{x}^{\dagger}(i,t) \right\|+\alpha_{t}U_{i,t}(1+J_{i,t}),
						\end{align}}
					which then falls under the purview of Lemma \ref{le:l1.1} and hence we have the assertion,
					{\small\begin{align}
						\label{eq:l2_17}
						\mathbb{P}\left(\lim_{t\to\infty}(t+1)^{\delta_0}\left(\widetilde{\mathbf{x}}(i,t)-\mathbf{1}_{Q_i}\otimes\widetilde{\mathbf{x}}_{\mbox{avg,i}}(t)\right)=0\right)=1,
						\end{align}}
					where $0<\delta_{0}<1-\tau_{1}$ and hence $\delta_{0}$ can be chosen to be $1/2 + \delta$, where $\delta > 0$ and we finally have,
					{\small\begin{align}
						\label{eq:l2_18}
						\mathbb{P}\left(\lim_{t\to\infty}(t+1)^{\frac{1}{2}+\delta}\left(\widetilde{\mathbf{x}}(t)-\mathbf{1}_{N}\otimes\widetilde{\mathbf{x}}_{\mbox{\scriptsize{avg}}}(t)\right)=0\right)=1,
						\end{align}}
					
					as the above analysis can be repeated each entry $i$ of the parameter of interest $\btheta^{\ast}$.
					
					The proof of Theorem \ref{th2} needs the following Lemma from \cite{Fabian-2} concerning the asymptotic normality of the stochastic recursions.
					\begin{Lemma}[Theorem 2.2 in \cite{Fabian-2}]
						\label{main_res_l0}
						\noindent Let $\{\mathbf{z}_{t}\}$ be an $\mathbb{R}^{k}$-valued $\{\mathcal{F}_{t}\}$-adapted process that satisfies
						{\small\begin{align}
							\label{eq:l0_1}
							&\mathbf{z}_{t+1}=\left(\mathbf{I}_{k}-\frac{1}{t+1}\Gamma_{t}\right)\mathbf{z}_{t}+(t+1)^{-1}\mathbf{\Phi}_{t}\mathbf{V}_{t}\nonumber\\&+(t+1)^{-3/2}\mathbf{T}_{t},
							\end{align}}
						\noindent where the stochastic processes $\{\mathbf{V}_{t}\}, \{\mathbf{T}_{t}\} \in \mathbb{R}^{k}$ while $\{\mathbf{\Gamma}_{t}\}, \{\mathbf{\Phi}_{t}\} \in \mathbb{R}^{k\times k}$. Moreover, suppose for each $t$, $\mathbf{V}_{t-1}$ and $\mathbf{T}_{t}$ are $\mathcal{F}_{t}$-adapted, whereas the processes $\{\mathbf{\Gamma}_{t}\}$, $\{\mathbf{\Phi}_{t}\}$ are $\{\mathcal{F}_{t}\}$-adapted.
						
						\noindent Also, assume that
						{\small\begin{align}
							\label{eq:l0_2}
							\mathbf{\Gamma}_{t}\to\mathbf{\Gamma}, \mathbf{\Phi}_{t}\to\mathbf{\Phi},~ \textit{and} ~\mathbf{T}_{t}\to 0 ~~\mbox{a.s. as $t\rightarrow\infty$},
							\end{align}}
						\noindent where $\mathbf{\Gamma}$ is a symmetric and positive definite matrix, and admits an eigen decomposition of the form $\mathbf{P}^{\top}\mathbf{\Gamma}\mathbf{P}=\mathbf{\Lambda}$, where $\mathbf{\Lambda}$ is a diagonal matrix and $\mathbf{P}$ is an orthogonal matrix. Furthermore, let the sequence $\{\mathbf{V}_{t}\}$ satisfy $\mathbb{E}\left[\mathbf{V}_{t}|\mathcal{F}_{t}\right]=0$ for each $t$ and suppose there exists a positive constant $C$ and a matrix $\Sigma$ such that $C > \left\|\mathbb{E}\left[\mathbf{V}_{t}\mathbf{V}_{t}^{\top}|\mathcal{F}_{t}\right]-\Sigma\right\|\to 0~a.s. ~\textit{as}~ t\to\infty$ and with $\sigma_{t,r}^{2}=\int_{\left\|\mathbf{V}_{t}\right\|^{2} \ge r(t+1)}\left\|\mathbf{V}_{t}\right\|^{2}d\mathbb{P}$, let $\lim_{t\to\infty}\frac{1}{t+1}\sum_{s=0}^{t}\sigma_{s,r}^{2}=0$ for every $r > 0$. Then, we have,
						{\small\begin{align}
							\label{eq:l0_3}
							(t+1)^{1/2}\mathbf{z}_{t}\overset{\mathcal{D}}{\Longrightarrow}\mathcal{N}\left(\mathbf{0}, \mathbf{P}\mathbf{M}\mathbf{P}^{\top}\right),
							\end{align}}
						\noindent where the $(i,j)$-th entry of the matrix $\mathbf{M}$ is given by
						{\small\begin{align}
							\label{eq:l0_4}
							\left[\mathbf{M}\right]_{ij}=\left[\mathbf{P}^{\top}\mathbf{\Phi}\mathbf{\Sigma}\mathbf{\Phi}^{\top}\mathbf{P}\right]_{ij}\left(\left[\mathbf{\Lambda}\right]_{ii}+\left[\mathbf{\Lambda}\right]_{jj}-1\right)^{-1}.
							\end{align}}
					\end{Lemma}
					Multiplying the selection matrix, we have,
					{\small\begin{align}
						\label{eq:t2_1}
						&\widetilde{\mathbf{x}}(i,t+1)=\widetilde{\mathbf{x}}(i,t)-\mathbf{L}_{\mathcal{P},i}(t)\widetilde{\mathbf{x}}(i,t)+\alpha_{t}\mathcal{S}_{i}\mathcal{P}\mathbf{G}_{H}\mathbf{R}^{-1}\nonumber\\&\times\left(\mathbf{y}(t)-\mathbf{G}_{H}^{\top}\mathcal{P}\widetilde{\mathbf{x}}(t)\right)\nonumber\\
						&\Rightarrow \frac{\mathbf{1}_{Q_{i}}^{\top}}{Q_{i}}\widetilde{\mathbf{x}}(i,t+1) = \frac{\mathbf{1}_{Q_{i}}^{\top}}{Q_{i}}\widetilde{\mathbf{x}}(i,t)-\frac{\mathbf{1}_{Q_{i}}^{\top}}{Q_{i}}\mathbf{L}_{\mathcal{P},i}(t)\widetilde{\mathbf{x}}(i,t)\nonumber\\&+\alpha_{t}\frac{\mathbf{1}_{Q_{i}}^{\top}}{Q_{i}}\mathcal{S}_{i}\mathcal{P}\mathbf{G}_{H}\mathbf{R}^{-1}\left(\mathbf{y}(t)-\mathbf{G}_{H}^{\top}\mathcal{P}\widetilde{\mathbf{x}}(t)\right)\nonumber\\
						&\Rightarrow \widetilde{\mathbf{x}}_{\mbox{\scriptsize{avg}},i}(t+1)=\widetilde{\mathbf{x}}_{\mbox{\scriptsize{avg}},i}(t)+\alpha_{t}\frac{\mathbf{1}_{Q_{i}}^{\top}}{Q_{i}}\mathcal{S}_{i}\mathcal{P}\mathbf{G}_{H}\mathbf{R}^{-1}\nonumber\\&\times\left(\mathbf{y}(t)-\mathbf{G}_{H}^{\top}\mathcal{P}\widetilde{\mathbf{x}}(t)\right),
						\end{align}}
					where $\{\widetilde{\mathbf{x}}_{\mbox{\scriptsize{avg}},i}(t)\}$ is the averaged estimate sequence for the $i$-th entry of the parameter $\btheta^{\ast}$. Stacking, all such averages together we have,
					{\small\begin{align}
						\label{eq:t2_2}
						&\widetilde{\mathbf{x}}_{\mbox{\scriptsize{avg}}}(t+1)=\widetilde{\mathbf{x}}_{\mbox{\scriptsize{avg}}}(t)+\alpha_{t}\mathcal{S}_{\mbox{\scriptsize{avg}}}\mathcal{P}\mathbf{G}_{H}\mathbf{R}^{-1}\left(\mathbf{y}(t)-\mathbf{G}_{H}^{\top}\mathcal{P}\widetilde{\mathbf{x}}(t)\right)\nonumber\\
						&\Rightarrow \widetilde{\mathbf{x}}_{\mbox{\scriptsize{avg}}}(t+1) - \btheta^{\ast} = \left(\mathbf{I}-\alpha_{t}\mathbf{Q}\sum_{n=1}^{N}\mathcal{P}_{\mathcal{I}_n}\mathbf{H}_{n}^{\top}\mathbf{R}^{-1}\mathbf{H}_{n}\mathcal{P}_{\mathcal{I}_n}\right)\nonumber\\&\times\left(\widetilde{\mathbf{x}}_{\mbox{\scriptsize{avg}}}(t)-\btheta^{\ast}\right)\nonumber\\
						&+\alpha_{t}\mathcal{S}_{\mbox{\scriptsize{avg}}}\mathcal{P}\mathbf{G}_{H}\mathbf{R}^{-1}\mathbf{\gamma}(t)\nonumber\\&+\alpha_{t}\mathbf{Q}\sum_{n=1}^{N}\mathcal{P}_{\mathcal{I}_n}\mathbf{H}_{n}^{\top}\mathbf{R}_{n}^{-1}\mathbf{H}_{n}\left(\widetilde{\mathbf{x}}_{n}(t)-\mathcal{P}_{\mathcal{I}_n}\widetilde{\mathbf{x}}_{\mbox{\scriptsize{avg}}}(t)\right),
						\end{align}}
					where {\small$\mathcal{S}_{\mbox{\scriptsize{avg}}}=\left[\frac{\mathbf{1}_{Q_{1}}^{\top}}{Q_{1}}\mathcal{S}_{1}, \frac{\mathbf{1}_{Q_{2}}^{\top}}{Q_{2}}\mathcal{S}_{2},\cdots,\frac{\mathbf{1}_{Q_{N}}^{\top}}{Q_{N}}\mathcal{S}_{N}\right]$ and $\mathbf{Q} = \textit{diag}\left[\frac{1}{Q_{1}},\frac{1}{Q_{2}},\cdots,\frac{1}{Q_{N}}\right]$.} In the above derivation, we make use of the fact that {\small$\mathcal{S}_{\mbox{\scriptsize{avg}}}\mathcal{P}\mathbf{G}_{H}\mathbf{R}^{-1}\mathbf{G}_{H}^{\top}\mathcal{P}\mathbf{1}_{N}\otimes\left(\widetilde{\mathbf{x}}_{\mbox{\scriptsize{avg}}}(t)-\btheta^{\ast}\right)=\mathbf{Q}\sum_{n=1}^{N}\mathcal{P}_{n}\mathbf{H}_{n}^{\top}\mathbf{R}^{-1}\mathbf{H}_{n}\left(\widetilde{\mathbf{x}}_{\mbox{\scriptsize{avg}}}(t)-\btheta^{\ast}\right)$}, which in turn follows from the fact that,
					{\small\begin{align}
						\label{eq:savg_explain}
						&\mathcal{S}_{\mbox{\scriptsize{avg}}} = \mathbf{Q}\left[\mathcal{P}_{\mathcal{I}_1}~\mathcal{P}_{\mathcal{I}_2}\cdots\mathcal{P}_{\mathcal{I}_N}\right]=\left[\mathbf{Q}\mathcal{P}_{\mathcal{I}_1}~\mathbf{Q}\mathcal{P}_{\mathcal{I}_2}\cdots\mathbf{Q}\mathcal{P}_{\mathcal{I}_N}\right]\nonumber\\
						&\Rightarrow \mathcal{S}_{\mbox{\scriptsize{avg}}}\mathcal{P}\mathbf{G}_{H}\mathbf{R}^{-1}\mathbf{G}_{H}^{\top}\mathcal{P}\mathbf{1}_{N}\otimes\left(\widetilde{\mathbf{x}}_{\mbox{\scriptsize{avg}}}(t)-\btheta^{\ast}\right) \nonumber\\&= \left[\mathbf{Q}\mathcal{P}_{\mathcal{I}_1}~\mathbf{Q}\mathcal{P}_{\mathcal{I}_2}\cdots\mathbf{Q}\mathcal{P}_{\mathcal{I}_N}\right]\nonumber\\&\times\left[\mathcal{P}_{\mathcal{I}_1}\mathbf{H}_{1}^{\top}\mathbf{R}_{1}^{-1}\mathbf{H}_{1}\left(\widetilde{\mathbf{x}}_{\mbox{\scriptsize{avg}}}(t)-\btheta^{\ast}\right)\right.\cdots\nonumber\\&\left.\mathcal{P}_{\mathcal{I}_N}\mathbf{H}_{N}^{\top}\mathbf{R}_{N}^{-1}\mathbf{H}_{N}\left(\widetilde{\mathbf{x}}_{\mbox{\scriptsize{avg}}}(t)-\btheta^{\ast}\right)\right]^{\top}\nonumber\\
						&=\mathbf{Q}\sum_{n=1}^{N}\mathcal{P}_{\mathcal{I}_n}\mathbf{H}_{n}^{\top}\mathbf{R}^{-1}\mathbf{H}_{n}\left(\widetilde{\mathbf{x}}_{\mbox{\scriptsize{avg}}}(t)-\btheta^{\ast}\right).
						\end{align}}
					
					\noindent Define, the residual sequence, $\{\mathbf{z}_{t}\}$, where $\mathbf{z}(t)=\widetilde{\mathbf{x}}_{\mbox{\scriptsize{avg}}}(t)-\btheta^{\ast}$, which can be then shown to satisfy the recursion
					{\small\begin{align}
						\label{eq:l1_2}
						\mathbf{z}_{t+1}=\left(\mathbf{I}_{N}-\alpha_{t}\Gamma\right)\mathbf{z}_{t}+\alpha_{t}\mathbf{U}_{t}+\alpha_{t}\mathbf{J}_{t},
						\end{align}}
					where
					{\small\begin{align}
						\label{eq:l1_3}
						&\Gamma = \mathbf{Q}\sum_{n=1}^{N}\mathcal{P}_{\mathcal{I}_n}\mathbf{H}_{n}^{\top}\mathbf{R}^{-1}\mathbf{H}_{n}\mathcal{P}_{\mathcal{I}_n}\nonumber\\
						&\mathbf{U}_{t}=\mathbf{Q}\sum_{n=1}^{N}\mathcal{P}_{\mathcal{I}_n}\mathbf{H}_{n}^{\top}\mathbf{R}_{n}^{-1}\mathbf{H}_{n}\left(\widetilde{\mathbf{x}}_{n}(t)-\mathcal{P}_{\mathcal{I}_n}\widetilde{\mathbf{x}}_{\mbox{\scriptsize{avg}}}(t)\right)\nonumber\\
						&\mathbf{J}_{t}=\mathcal{S}_{\mbox{\scriptsize{avg}}}\mathcal{P}\mathbf{G}_{H}\mathbf{R}^{-1}\mathbf{\gamma}(t).
						\end{align}}
					\noindent We rewrite the recursion for $\{\mathbf{z}_{t}\}$ as follows:
					{\small\begin{align}
						\label{eq:l3_2}
						\mathbf{z}_{t+1} = \left(\mathbf{I}_{N}-\alpha_{t}\Gamma_{t}\right)\mathbf{z}_{t}+(t+1)^{-3/2}\mathbf{T}_{t}+(t+1)^{-1}\mathbf{\Phi}_{t}\mathbf{V}_{t},
						\end{align}}
					where
					{\small\begin{align}
						\label{eq:l3_3}
						&\mathbf{\Gamma}_{t} = \mathbf{\Gamma} =\mathbf{Q}\sum_{n=1}^{N}\mathcal{P}_{\mathcal{I}_n}\mathbf{H}_{n}^{\top}\mathbf{R}^{-1}\mathbf{H}_{n}\mathcal{P}_{\mathcal{I}_n}, \mathbf{\Phi}_{t} = a\mathbf{I}\nonumber\\
						&\mathbf{T}_{t}=a(t+1)^{1/2}\mathbf{U}_{t}\nonumber\\& = a\mathbf{Q}\sum_{n=1}^{N}\mathcal{P}_{\mathcal{I}_n}\mathbf{H}_{n}^{\top}\mathbf{R}_{n}^{-1}\mathbf{H}_{n}(t+1)^{0.5}\left(\widetilde{\mathbf{x}}_{n}(t)-\mathcal{P}_{\mathcal{I}_n}\widetilde{\mathbf{x}}_{\mbox{\scriptsize{avg}}}(t)\right)\xrightarrow{t\to\infty}0\nonumber\\
						&\mathbf{V}_{t}=\mathbf{J}_{t}=\mathcal{S}_{\mbox{\scriptsize{avg}}}\mathcal{P}\mathbf{G}_{H}\mathbf{R}^{-1}\mathbf{\gamma}(t),~\mathbb{E}\left[\mathbf{V}_{t}|\mathcal{F}_{t}\right]=0,\nonumber\\&\mathbb{E}\left[\mathbf{V}_{t}\mathbf{V}^{\top}_{t}|\mathcal{F}_{t}\right]= \mathcal{S}_{\mbox{\scriptsize{avg}}}\mathcal{P}\mathbf{G}_{H}\mathbf{R}^{-1}\mathbf{G}_{H}^{\top}\mathcal{P}\mathcal{S}_{\mbox{\scriptsize{avg}}}\nonumber\\& = \mathbf{Q}\left(\sum_{n=1}^{N}\mathcal{P}_{\mathcal{I}_n}\mathbf{H}_{n}^{\top}\mathbf{R}^{-1}\mathbf{H}_{n}\mathcal{P}_{\mathcal{I}_n}\right)\mathbf{Q}
						\end{align}}
					\noindent Due to the i.i.d nature of the noise process, we have the uniform integrability condition for the process $\{\mathbf{V}_{t}\}$. Hence, $\{\mathbf{x}_{\mbox{\scriptsize{avg}}}(t)\}$ falls under the purview of Lemma \ref{main_res_l0} and we thus conclude that
					{\small\begin{align}
						\label{eq:l3_4}
						(t+1)^{1/2}\left(\widetilde{\mathbf{x}}_{\mbox{\scriptsize{avg}}}(t)-\btheta\right)\overset{\mathcal{D}}{\Longrightarrow}\mathcal{N}(0,\mathbf{P}\mathbf{M}\mathbf{P}^{\top}),
						\end{align}}
					in which,
					{\small\begin{align}
						\label{eq:th2_21}
						&\left[\mathbf{M}\right]_{ij}=\left[\mathbf{P}\mathbf{Q}\left(\sum_{n=1}^{N}\mathcal{P}_{\mathcal{I}_n}\mathbf{H}_{n}^{\top}\mathbf{R}_{n}^{-1}\mathbf{H}_{n}\mathcal{P}_{\mathcal{I}_n}\right)\mathbf{Q}\mathbf{P}\right]_{ij}\nonumber\\&\times\left(\left[\mathbf{\Lambda}\right]_{ii}+\left[\mathbf{\Lambda}\right]_{jj}-1\right)^{-1},
						\end{align}}
					where $\mathbf{P}$ and $\mathbf{\Lambda}$ are orthonormal and diagonal matrices such that $\mathbf{P}^{\top}\mathbf{Q}\left(\sum_{n=1}^{N}\mathcal{P}_{\mathcal{I}_n}\mathbf{H}_{n}^{\top}\mathbf{R}_{n}^{-1}\mathbf{H}_{n}\mathcal{P}_{\mathcal{I}_n}\right)\mathbf{Q}\mathbf{P} = \mathbf{\Lambda}$.
					\noindent Now from \eqref{eq:l2_18}, we have that the processes $\{\widetilde{\mathbf{x}}_{n}(t)\}$ and $\{\widetilde{\mathbf{x}}_{\mbox{\scriptsize{avg}}}(t)\}$ are indistinguishable in the $(t+1)^{1/2}$ time scale, which is formalized as follows:
					{\small\begin{align}
						\label{eq:l3_7}
						&\mathbb{P}_{\btheta}\left(\lim_{t\to\infty}\left\|\sqrt{t+1}\left(\widetilde{\mathbf{x}}(t)-\btheta\right)-\sqrt{t+1}\left(\widetilde{\mathbf{x}}_{\mbox{\scriptsize{avg}}}(t)-\btheta\right)\right\|=0\right)\nonumber\\
						&=\mathbb{P}_{\btheta}\left(\lim_{t\to\infty}\left\|\sqrt{t+1}\left(\widetilde{\mathbf{x}}(t)-\widetilde{\mathbf{x}}_{\mbox{\scriptsize{avg}}}(t)\right)\right\|=0\right)=1.
						\end{align}}
					\noindent Thus, the difference of the sequences $\left\{\sqrt{t+1}\left(\widetilde{\mathbf{x}}_{n}(t)-\btheta\right)\right\}$ and $\left\{\sqrt{t+1}\left(\widetilde{\mathbf{x}}_{\mbox{\scriptsize{avg}}}(t)-\btheta\right)\right\}$ converges a.s. to zero as $t\rightarrow\infty$ \noindent and hence we have,
					{\small\begin{align}
						\label{eq:t2_pr_12}
						\sqrt{t+1}\left(\widetilde{\mathbf{x}}_{n}(t)-\btheta\right)\overset{\mathcal{D}}{\Longrightarrow}\mathcal{N}(0,\mathbf{P}\mathbf{M}\mathbf{P}^{\top}).
						\end{align}}	
				\end{proof}
				\section{CONCLUSION}
				\label{sec:conc}
				\noindent In this paper, we have proposed a $\emph{consensus}+\emph{innovations}$ type algorithm, $\mathcal{CIRFE}$, for estimating a high-dimensional parameter or field that exhibits a cyber-physical flavor. In the proposed algorithm, every agent updates its estimate of a few components of the high-dimensional parameter vector by simultaneous processing of neighborhood information and local newly sensed information and in which the inter-agent collaboration is restricted to a possibly sparse communication graph. Under rather generic assumptions we establish the consistency of the parameter estimate sequence and characterize the asymptotic variance of the proposed estimator. A natural direction for future research consists of considering models with non-linear observation functions and extension of the proposed algorithm $\mathcal{CIRFE}$ to quantized communication schemes in the lines of \cite{kar2012distributed} and \cite{Zhang2017DistributedDO}.

				\bibliographystyle{IEEEtran}
				\bibliography{CentralBib,dsprt,glrt}

% Generated by IEEEtran.bst, version: 1.14 (2015/08/26)
\begin{thebibliography}{10}
\providecommand{\url}[1]{#1}
\csname url@samestyle\endcsname
\providecommand{\newblock}{\relax}
\providecommand{\bibinfo}[2]{#2}
\providecommand{\BIBentrySTDinterwordspacing}{\spaceskip=0pt\relax}
\providecommand{\BIBentryALTinterwordstretchfactor}{4}
\providecommand{\BIBentryALTinterwordspacing}{\spaceskip=\fontdimen2\font plus
\BIBentryALTinterwordstretchfactor\fontdimen3\font minus
  \fontdimen4\font\relax}
\providecommand{\BIBforeignlanguage}[2]{{%
\expandafter\ifx\csname l@#1\endcsname\relax
\typeout{** WARNING: IEEEtran.bst: No hyphenation pattern has been}%
\typeout{** loaded for the language `#1'. Using the pattern for}%
\typeout{** the default language instead.}%
\else
\language=\csname l@#1\endcsname
\fi
#2}}
\providecommand{\BIBdecl}{\relax}
\BIBdecl

\bibitem{ahuja1988network}
R.~K. Ahuja, T.~L. Magnanti, and J.~B. Orlin, ``Network flows,'' 1988.

\bibitem{mota2015distributed}
J.~F. Mota, J.~M. Xavier, P.~M. Aguiar, and M.~P{\"u}schel, ``Distributed
  optimization with local domains: Applications in mpc and network flows,''
  \emph{IEEE Transactions on Automatic Control}, vol.~60, no.~7, pp.
  2004--2009, 2015.

\bibitem{halvgaard2016distributed}
R.~Halvgaard, L.~Vandenberghe, N.~K. Poulsen, H.~Madsen, and J.~B.
  J{\o}rgensen, ``Distributed model predictive control for smart energy
  systems,'' \emph{IEEE Transactions on Smart Grid}, vol.~7, no.~3, pp.
  1675--1682, 2016.

\bibitem{Khan-DILAND-TSP-2010}
U.~A. Khan, S.~Kar, and J.~M.~F. Moura, ``{DILAND}: An algorithm for
  distributed sensor localization with noisy distance measurements,''
  \emph{IEEE Transactions on Signal Processing}, vol.~58, no.~3, pp. 1940 --
  1947, March 2010.

\bibitem{kekatos2013distributed}
V.~Kekatos and G.~B. Giannakis, ``Distributed robust power system state
  estimation,'' \emph{IEEE Transactions on Power Systems}, vol.~28, no.~2, pp.
  1617--1626, 2013.

\bibitem{de2010synchronized}
J.~De~La~Ree, V.~Centeno, J.~S. Thorp, and A.~G. Phadke, ``Synchronized phasor
  measurement applications in power systems,'' \emph{IEEE Transactions on smart
  grid}, vol.~1, no.~1, pp. 20--27, 2010.

\bibitem{kar2011convergence}
S.~Kar and J.~M. Moura, ``Convergence rate analysis of distributed gossip
  (linear parameter) estimation: Fundamental limits and tradeoffs,'' \emph{IEEE
  Journal of Selected Topics in Signal Processing}, vol.~5, no.~4, pp.
  674--690, 2011.

\bibitem{kar2013consensus+}
------, ``Consensus+ innovations distributed inference over networks:
  cooperation and sensing in networked systems,'' \emph{IEEE Signal Processing
  Magazine}, vol.~30, no.~3, pp. 99--109, 2013.

\bibitem{Mesbahi-parameter}
A.~Das and M.~Mesbahi, ``Distributed linear parameter estimation in sensor
  networks based on {L}aplacian dynamics consensus algorithm,'' in \emph{3rd
  Annual IEEE Communications Society on Sensor and Ad Hoc Communications and
  Networks}, vol.~2, Reston, VA, USA, 28-28 Sept. 2006, pp. 440--449.

\bibitem{Giannakis-est}
I.~D. Schizas, A.~Ribeiro, and G.~B. Giannakis, ``Consensus in {A}d {H}oc
  {WSN}s with noisy links - part {I}: Distributed estimation of deterministic
  signals,'' \emph{IEEE Transactions on Signal Processing}, vol.~56, no.~1, pp.
  350--364, January 2008.

\bibitem{Sayed-LMS}
C.~G. Lopes and A.~H. Sayed, ``Diffusion least-mean squares over adaptive
  networks: Formulation and performance analysis,'' \emph{IEEE Transactions on
  Signal Processing}, vol.~56, no.~7, pp. 3122--3136, July 2008.

\bibitem{Stankovic-parameter}
S.~Stankovic, M.~Stankovic, and D.~Stipanovic, ``Decentralized parameter
  estimation by consensus based stochastic approximation,'' in \emph{46th IEEE
  Conference on Decision and Control}, New Orleans, LA, USA, 12-14 Dec. 2007,
  pp. 1535--1540.

\bibitem{Giannakis-LMS}
I.~Schizas, G.~Mateos, and G.~Giannakis, ``Stability analysis of the
  consensus-based distributed {LMS} algorithm,'' in \emph{Proceedings of the
  33rd International Conference on Acoustics, Speech, and Signal Processing},
  Las Vegas, Nevada, USA, April 1-4 2008, pp. 3289--3292.

\bibitem{Nedic-parameter}
S.~Ram, V.~Veeravalli, and A.~Nedic, ``Distributed and recursive parameter
  estimation in parametrized linear state-space models,'' \emph{IEEE
  Transactions on Automatic Control}, vol.~55, no.~2, pp. 488-- 492, February
  2010.

\bibitem{sahu2016distributed}
A.~K. Sahu and S.~Kar, ``Distributed sequential detection for {G}aussian
  shift-in-mean hypothesis testing,'' \emph{IEEE Transactions on Signal
  Processing}, vol.~64, no.~1, pp. 89--103, 2016.

\bibitem{nevelson1973stochastic}
M.~B. Nevelson and R.~Z. Khasʹminski{\u\i}, \emph{Stochastic approximation and
  recursive estimation}.\hskip 1em plus 0.5em minus 0.4em\relax American
  Mathematical Society, 1973, vol.~47.

\bibitem{Bertsekas-survey}
D.~Bertsekas, J.~Tsitsiklis, and M.~Athans, ``Convergence theories of
  distributed iterative processes: A survey,'' \emph{Technical Report for
  Information and Decision Systems, Massachusetts Inst. of Technology,
  Cambridge, MA}, 1984.

\bibitem{olfatisaberfaxmurray07}
R.~Olfati-Saber, J.~A. Fax, and R.~M. Murray, ``Consensus and cooperation in
  networked multi-agent systems,'' \emph{Proceedings of the IEEE}, vol.~95,
  no.~1, pp. 215--233, January 2007.

\bibitem{jadbabailinmorse03}
A.~Jadbabaie, J.~Lin, and A.~S. Morse, ``Coordination of groups of mobile
  autonomous agents using nearest neighbor rules,'' \emph{IEEE Transactions on
  Automatic Control}, vol.~48, no.~6, pp. 988--1001, Jun. 2003.

\bibitem{nedic2015nonasymptotic}
A.~Nedi{\'c}, A.~Olshevsky, and C.~A. Uribe, ``Nonasymptotic convergence rates
  for cooperative learning over time-varying directed graphs,'' in
  \emph{American Control Conference (ACC), 2015}.\hskip 1em plus 0.5em minus
  0.4em\relax IEEE, 2015, pp. 5884--5889.

\bibitem{jadbabaie2012non}
A.~Jadbabaie, P.~Molavi, A.~Sandroni, and A.~Tahbaz-Salehi, ``Non-{B}ayesian
  social learning,'' \emph{Games and Economic Behavior}, vol.~76, no.~1, pp.
  210--225, 2012.

\bibitem{shahrampour2013exponentially}
S.~Shahrampour and A.~Jadbabaie, ``Exponentially fast parameter estimation in
  networks using distributed dual averaging,'' in \emph{Decision and Control
  (CDC), 2013 IEEE 52nd Annual Conference on}.\hskip 1em plus 0.5em minus
  0.4em\relax IEEE, 2013, pp. 6196--6201.

\bibitem{mateos2009distributed}
G.~Mateos, I.~D. Schizas, and G.~B. Giannakis, ``Distributed recursive
  least-squares for consensus-based in-network adaptive estimation,''
  \emph{IEEE Transactions on Signal Processing}, vol.~57, no.~11, pp.
  4583--4588, 2009.

\bibitem{mateos2012distributed}
G.~Mateos and G.~B. Giannakis, ``Distributed recursive least-squares: Stability
  and performance analysis,'' \emph{IEEE Transactions on Signal Processing},
  vol.~60, no.~7, pp. 3740--3754, 2012.

\bibitem{weng2014efficient}
Z.~Weng and P.~M. Djuri{\'c}, ``Efficient estimation of linear parameters from
  correlated node measurements over networks,'' \emph{IEEE Signal Processing
  Letters}, vol.~21, no.~11, pp. 1408--1412, 2014.

\bibitem{cattivelli2010diffusion}
F.~S. Cattivelli and A.~H. Sayed, ``Diffusion lms strategies for distributed
  estimation,'' \emph{IEEE Transactions on Signal Processing}, vol.~58, no.~3,
  pp. 1035--1048, 2010.

\bibitem{cattivelli2008diffusion}
F.~S. Cattivelli, C.~G. Lopes, and A.~H. Sayed, ``Diffusion recursive
  least-squares for distributed estimation over adaptive networks,'' \emph{IEEE
  Transactions on Signal Processing}, vol.~56, no.~5, pp. 1865--1877, 2008.

\bibitem{Khan-Moura}
U.~A. Khan and J.~M.~F. Moura, ``Distributing the {K}alman filter for
  large-scale systems,'' \emph{IEEE Transactions on Signal Processing},
  vol.~56, no.~10, p. 4919�4935, October 2008.

\bibitem{das2015distributed}
S.~Das and J.~M.~F. Moura, ``Distributed {K}alman filtering with dynamic
  observations consensus,'' \emph{IEEE Transactions on Signal Processing},
  vol.~63, no.~17, pp. 4458--4473, 2015.

\bibitem{bogdanovic2014distributed}
N.~Bogdanovic, J.~Plata-Chaves, and K.~Berberidis, ``Distributed
  diffusion-based lms for node-specific parameter estimation over adaptive
  networks,'' in \emph{Acoustics, Speech and Signal Processing (ICASSP), 2014
  IEEE International Conference on}.\hskip 1em plus 0.5em minus 0.4em\relax
  IEEE, 2014, pp. 7223--7227.

\bibitem{nassif2017diffusion}
R.~Nassif, C.~Richard, A.~Ferrari, and A.~H. Sayed, ``Diffusion {LMS} for
  multitask problems with local linear equality constraints,'' \emph{IEEE
  Transactions on Signal Processing}, vol.~65, no.~19, pp. 4979--4993, 2017.

\bibitem{kar2010large}
\BIBentryALTinterwordspacing
S.~Kar, ``Large scale networked dynamical systems: Distributed inference,''
  Ph.D. dissertation, Carnegie Mellon University, Pittsburgh, PA, 2010.
  [Online]. Available: \url{http://gradworks.umi.com/34/21/3421734.html}
\BIBentrySTDinterwordspacing

\bibitem{chung1997spectral}
F.~R. Chung, \emph{Spectral graph theory}.\hskip 1em plus 0.5em minus
  0.4em\relax American Mathematical Soc., 1997, vol.~92.

\bibitem{kar2013distributed}
S.~Kar, J.~M.~F. Moura, and H.~V. Poor, ``Distributed linear parameter
  estimation: Asymptotically efficient adaptive strategies,'' \emph{SIAM
  Journal on Control and Optimization}, vol.~51, no.~3, pp. 2200--2229, 2013.

\bibitem{sahu2016distributedtsipn}
A.~K. Sahu, S.~Kar, J.~M. Moura, and H.~V. Poor, ``Distributed constrained
  recursive nonlinear least-squares estimation: Algorithms and asymptotics,''
  \emph{IEEE Transactions on Signal and Information Processing over Networks},
  vol.~2, no.~4, pp. 426--441, 2016.

\bibitem{alghunaim2017distributed}
S.~A. Alghunaim and A.~H. Sayed, ``Distributed coupled multi-agent stochastic
  optimization,'' \emph{arXiv preprint arXiv:1712.08817}, 2017.

\bibitem{Borkar-stochapp}
V.~S. Borkar, \emph{Stochastic Approximation: A Dynamical Systems
  Viewpoint}.\hskip 1em plus 0.5em minus 0.4em\relax Cambridge, UK: Cambridge
  University Press, 2008.

\bibitem{Gelfand-Mitter}
S.~B. Gelfand and S.~K. Mitter, ``Recursive stochastic algorithms for global
  optimization in $\mathbb{R}^d$,'' \emph{{SIAM} J. Control Optim.}, vol.~29,
  no.~5, pp. 999--1018, September 1991.

\bibitem{Fabian-2}
V.~Fabian, ``On asymptotic normality in stochastic approximation,'' \emph{The
  Annals of Mathematical Statistics}, vol.~39, no.~4, pp. 1327--1332, August
  1968.

\bibitem{kar2012distributed}
S.~Kar, J.~M. Moura, and K.~Ramanan, ``Distributed parameter estimation in
  sensor networks: Nonlinear observation models and imperfect communication,''
  \emph{IEEE Transactions on Information Theory}, vol.~58, no.~6, pp.
  3575--3605, 2012.

\bibitem{Zhang2017DistributedDO}
J.~Zhang, K.~You, and T.~Basar, ``Distributed discrete-time optimization in
  multi-agent networks using only sign of relative state,'' \emph{arXiv
  preprint arXiv:1709.08360}, 2017.

\end{thebibliography}

			\end{document}